\documentclass[final,hidelinks,onefignum,onetabnum]{siamart250211}

\usepackage{booktabs}        
\usepackage{mathtools}       
\usepackage{microtype}       
\usepackage{graphicx}        
\usepackage{enumitem}        
\usepackage{url}
\usepackage{xcolor}
\usepackage{amsfonts}
\usepackage{algorithm}
\usepackage{algorithmic}
\usepackage{amssymb}

\providecommand{\citet}[1]{\cite{#1}}
\providecommand{\citep}[1]{\cite{#1}}
 
\providecommand{\citeauthor}[1]{\textsc{#1}}
\providecommand{\citeyear}[1]{\cite{#1}}

\providecommand{\R}{\mathbb{R}}

\providecommand{\tr}{\mathrm{tr}}
\providecommand{\diag}{\mathrm{diag}}
\providecommand{\1}{\mathbf{1}}

\graphicspath{{./}}

\newtheorem{rmk}{Remark}[section]

\theoremstyle{plain}

\DeclareMathOperator{\rank}{rank}

\headers{Beyond Low Rank}{K. Yeon, and M. Anitescu}

\title{Beyond Low Rank: Fast Low-Rank + Diagonal Decomposition with a Spectral Approach}

\author{
  Kingsley Yeon\thanks{Department of Statistics, University of Chicago
  (\email{yeon@uchicago.edu}).}
  \and
  Mihai Anitescu\thanks{Mathematics and Computer Science, Argonne National Laboratory
  (\email{anitescu@mcs.anl.gov}).}
}

\begin{document}
\maketitle

\begin{abstract}
Low-rank plus diagonal (LRPD) decompositions provide a powerful structural model for large covariance matrices, simultaneously capturing global shared factors and localized corrections that arise in covariance estimation, factor analysis, and large-scale kernel learning. We introduce an alternating low-rank then diagonal (Alt) algorithm that provably reduces approximation error and significantly outperforms gradient descent while remaining cheaper than majorization–minimization methods~\cite{sun2016majorization}. To scale to large matrices, we develop a randomized LRPD variant that combines fixed-rank Nyström sketching~\cite{tropp2017fixed} for the low-rank component with Diag++ stochastic diagonal estimation~\cite{baston2022stochastic}. This hybrid algorithm achieves machine precision decomposition error using a number of matrix–vector products far smaller than the ambient dimension, and comes with rigorous non-asymptotic error bounds. On synthetic data, it exactly recovers LRPD structured matrices with high efficiency, and on real-world S\&P 500 stock return covariances, where the spectrum decays slowly and strong sector structure exists, it achieves substantially lower error than pure low-rank approximations.
\end{abstract}



\section{Introduction}
In this work, we revisit the low-rank plus diagonal (LRPD) structure from both a theoretical and
algorithmic perspective.  LRPD models provide a natural way to capture global correlations through a
low-rank component while accounting for strong variance patterns specific to each variable via a
diagonal correction.  This simple yet expressive structure arises in many settings, including
covariance estimation~\cite{ledoit2004well, bickel2008regularized}, factor
analysis~\cite{lawley1940factor, anderson2003introduction}, and high-dimensional
statistics~\cite{fan2013large, johnstone2001distribution}, where a few latent factors explain most
of the shared variability while each dimension retains substantial self-behavior.

There is a rich body of work exploring diagonal-plus-low-rank decompositions. Stein~\cite{stein2014limitations} showed that covariance matrices can be effectively approximated by a diagonal matrix plus a low-rank correction, though his analysis was restricted to cases where the diagonal is a scalar multiple of the identity. Further developments such as the majorization–minimization approach~\cite{sun2016majorization} and convex relaxations~\cite{zhou2022covariance} improved numerical methods for related formulations, while applications in asset covariance estimation demonstrated superior empirical performance of LRPD models over structured alternatives like HOLDR~\cite{parshakova2024factor}. Similar decompositions also appear in robust PCA and structured matrix inference~\cite{bonnabel2024low, saunderson2012diagonal, zhao2016low}.

We introduce a spectral perspective on low-rank plus diagonal (LRPD) approximation that leads to a new alternating algorithm with rigorous theoretical guarantees, including provable convergence. A key feature of this formulation is that, under generic assumptions, it renders the problem well-posed: the low-rank component is uniquely defined up to the ordering of the eigenvectors, despite the underlying nonconvexity. Beginning with a truncated eigen-decomposition, we show that an appropriate diagonal adjustment strictly improves operator-norm error relative to pure low-rank approximation, while direct diagonal subtraction fails to preserve positive semidefiniteness. Guided by this analysis, we propose an alternating low-rank–diagonal scheme based on successive spectral projections, which eliminates the rotational degeneracy inherent in factorized representations. We establish monotone decrease of the approximation error and prove a spectral gap-free local contraction result that depends only on the geometry of the eigenspace.

To address scalability in large problems, we further introduce a randomized LRPD variant that integrates Nyström sketching \cite{frangella2023randomized} for the low-rank component with stochastic diagonal estimation \cite{baston2022stochastic}. This approach achieves high accuracy while requiring significantly fewer matrix--vector products and reduced memory. We demonstrate the effectiveness of our methods on both synthetic data and real-world covariance matrices, including clustered S\&P 500 stock return covariances. In this setting, broad market factors such as interest rates act as a shared low-rank component that influences all stocks simultaneously. Although the S\&P 500 covariance matrix is not low-rank due to its slowly decaying spectrum, it exhibits a clear clustered structure. We show that the LRPD model can be generalized to a block-diagonal form using the same alternating algorithm, for which we do not pursue a theoretical analysis here, and that this extension yields improved fits by capturing both global market effects and sector-specific behaviors.

The remainder of the paper is organized as follows.  Section~2 introduces the alternating low
rank–diagonal (Alt) decomposition and establishes its local convergence properties.  Section~3
develops randomized variants of Alt for large scale settings where only matrix–vector products are
available.  Section~4 presents applications to covariance estimation and financial datasets such as
the S\&P\,500.  Section~5 compares Alt with naive gradient descent by deriving the exact first order
updates and demonstrating the substantial performance gap between the two approaches.


\section{Low-rank plus diagonal (LRPD) decomposition}

This section develops a foundational perspective on low-rank plus diagonal (LRPD) decompositions by starting from the simplest spectral construction and analyzing its error properties. We first show how a truncated eigen-decomposition combined with a diagonal correction yields a guaranteed improvement over a pure low-rank approximation in operator norm. Through explicit examples, we highlight why subtracting the diagonal directly fails to preserve positive semidefiniteness, motivating the need for a more structured approach. Building on this, we explore how such naive constructions behave when the target matrix is itself exactly LRPD, and why additional iterations or refined factorizations are needed to fully recover its structure. The following proposition establishes the basic existence and error reduction guarantees that serve as the starting point for these developments.

\begin{proposition}[Existence and Error Reduction in a Na\"ive Diagonal + Low-Rank Approximation] \label{prop:naivedecomposition}
Let \(\Sigma\in\R^{n\times n}\) be symmetric positive semidefinite with eigen-decomposition
\[
  \Sigma = V\Lambda V^T,\quad
  \Lambda=\operatorname{diag}(\lambda_1,\dots,\lambda_n),\quad
  \lambda_1\ge\cdots\ge\lambda_n\ge0.
\]
Fix \(0\le k\le n\), and write \(V=[V_k\;\;V_{>k}]\), \(\Lambda=\mathrm{diag}(\Lambda_k,\Lambda_{>k})\).  Define
\[
  U_k = V_k\,\Lambda_k^{1/2}, 
  \qquad
  D_k = \diag\!\bigl(\Sigma_{ii}-(U_kU_k^T)_{ii}\bigr),
\]
and set the residual
\[
  S_k = \Sigma - U_kU_k^T
  = V_{>k}\,\Lambda_{>k}\,V_{>k}^T \succeq 0,
\]
and the “corrected’’ residual
\[
  R_k = \Sigma - (D_k + U_kU_k^T)
  = S_k \;-\;D_k.
\]
Then:
\begin{enumerate}
  \item \(\Sigma = D_k + U_kU_k^T + R_k\), with \(D_k\succeq0\).
  \item \(\|S_k\|_2 = \lambda_{k+1}\) and
  \(\|R_k\|_2 \le \lambda_{k+1}\), with strict inequality
  \(\|R_k\|_2<\lambda_{k+1}\) whenever \(\lambda_{k+1}>0\).  
  \item If \(\lambda_{k+1}=0\) (i.e.\ \(\rank\Sigma\le k\)), then \(D_k=0\) and \(R_k=0\), so \(\Sigma=U_kU_k^T\).
\end{enumerate}
In particular, whenever \(\Sigma\) is full-rank beyond \(k\), including \(D_k\) strictly reduces the operator-norm error compared to the pure rank-\(k\) approximation.
\end{proposition}

\begin{proof}
By construction \(U_kU_k^T=V_k\Lambda_kV_k^T\), so 
\[
S_k=\Sigma-U_kU_k^T=V_{>k}\Lambda_{>k}V_{>k}^T,\quad
\|S_k\|_2=\lambda_{k+1}.
\]
Since \((U_kU_k^T)_{ii}\le\Sigma_{ii}\), we have \(D_k\succeq0\).  
Moreover, recall that \(R_k = S_k - \diag(S_k)\).  
To show \(\|R_k\|_2 \le \|S_k\|_2\), note that for any unit vector \(x\),
\[
|x^\top R_k x| = |x^\top S_k x - x^\top \diag(S_k) x|
  \le \max\{x^\top S_k x,\; x^\top \diag(S_k) x\}.
\]
Taking the supremum over all \(\|x\|_2=1\) gives
\[
\|R_k\|_2 = \sup_{\|x\|=1} |x^\top R_k x|
   \le \max\{\|S_k\|_2,\|\diag(S_k)\|_2\}.
\]
Since \(S_k \succeq 0\), the diagonal matrix \(D = \diag(S_k)\) is also positive semidefinite.  
For a diagonal matrix, its spectral norm equals its largest diagonal entry in absolute value, and since \(S_k\) is positive semidefinite, these entries are nonnegative. Hence
\[
\|D\|_2 = \max_i D_{ii} = \max_i (S_k)_{ii}.
\]
Each diagonal entry satisfies \((S_k)_{ii} = e_i^\top S_k e_i\), where \(e_i\) is the \(i\)-th standard basis vector.  
Because for any unit vector \(x\), \(x^\top S_k x \le \lambda_{\max}(S_k)\), we have
\[
(S_k)_{ii} = e_i^\top S_k e_i \le \lambda_{\max}(S_k) \quad \forall i.
\]
Taking the maximum over \(i\) yields
\[
\|D\|_2 = \max_i (S_k)_{ii} = \max_i e_i^\top S_k e_i \le \lambda_{\max}(S_k) = \|S_k\|_2.
\]
we conclude that
\[
\|R_k\|_2 \le \|S_k\|_2.
\]

Moreover, if \(\lambda_{k+1}>0\), then \(\diag(S_k)\neq0\), so subtracting it strictly lowers the Rayleigh quotient of the top eigenvector of \(S_k\) and hence \(\|R_k\|_2<\lambda_{k+1}\).  Finally, if \(\lambda_{k+1}=0\) then \(S_k=0\), forcing \(D_k=0\) and \(R_k=0\).
\end{proof}

\begin{rmk}
    Consider a naive decomposition where $D = \text{diag}(\Sigma)$. An issue arises when decomposing $\Sigma$ as $\Sigma = D + UU^{T}$ by subtracting the diagonal part $D = \text{diag}(\Sigma)$ and expressing the residual $R := \Sigma - D$ as $UU^{T}$, rather than performing the low-rank factorization first as in Proposition \ref{prop:naivedecomposition}. While $UU^{T}$ is positive semidefinite, the residual $R$ need not be. Specifically, $R$ has a zero diagonal by construction, and symmetric matrices with a zero diagonal can have both positive and negative eigenvalues. A simple counterexample illustrates this:
    \[
        \Sigma = \begin{pmatrix}2 & 1\\[3pt]1 & 2\end{pmatrix},
        \quad
        D = \diag(\Sigma) = \begin{pmatrix}2 & 0\\[3pt]0 & 2\end{pmatrix},
        \quad
        R_{\mathrm{naive}} = \Sigma - D
        = \begin{pmatrix}0 & 1\\[3pt]1 & 0\end{pmatrix}.
    \]
    The eigenvalues of $R_{\mathrm{naive}}$ are $\pm1$, so
    \[
      \|R_{\mathrm{naive}}\|_{2} \;=\; 1,
    \]
    and $R_{\mathrm{naive}}$ cannot be written as $UU^{T}$.

    By contrast, using the decomposition in Proposition \ref{prop:naivedecomposition} with $k=1$ yields the spectral splitting
    \[
      \lambda_{1}=3,\quad
      v_{1}=\frac1{\sqrt2}\begin{pmatrix}1\\1\end{pmatrix},
      \quad
      U_{1}=v_{1}\sqrt{\lambda_{1}}
            =\begin{pmatrix}\sqrt{3/2}\\[3pt]\sqrt{3/2}\end{pmatrix},
    \]
    \[
      D^{(1)} \;=\;\diag\!\Bigl(2-(U_{1}U_{1}^{T})_{11},\;2-(U_{1}U_{1}^{T})_{22}\Bigr)
             = \diag\bigl(\tfrac12,\tfrac12\bigr),
    \]
    \[
      R^{(1)}
      = \Sigma \;-\;\bigl(D^{(1)} + U_{1}U_{1}^{T}\bigr)
      = \begin{pmatrix}0 & -\tfrac12\\[3pt]-\tfrac12 & 0\end{pmatrix},
    \]
    whose eigenvalues are $\pm\tfrac12$.  Hence
    \[
      \|R^{(1)}\|_{2}
      = \tfrac12,
    \]
    which is strictly smaller than the naive residual norm of 1.  This demonstrates that using the top eigenpair to form the rank‐1 term yields a much smaller operator‐norm error than simply subtracting $\diag(\Sigma)$.
\end{rmk}

The construction in Proposition \ref{prop:naivedecomposition} addresses the existence of such a decomposition and shows that when the matrix we want to approximate is low-rank, the Eckart–Young theorem forces \( D = 0 \). However, when the matrix is full-rank but still structured, can we still obtain a diagonal + low-rank decomposition with controlled approximation error? And when the matrix itself is low-rank, can we improve upon the corrected residual provided by Proposition~\ref{prop:naivedecomposition}?

To illustrate the first question, consider a simple example showing that even if the matrix is exactly low-rank plus diagonal (LRPD), our naive decomposition approach may fail to recover it. Take, for instance, the case where \(\Sigma\) is the identity plus a rank-one matrix:

\[
u = \frac{1}{\sqrt{2}}\begin{pmatrix}1\\1\end{pmatrix},
\qquad
\Sigma = I_2 + uu^{T} = \begin{pmatrix}1 & 0\\ 0 & 1\end{pmatrix} + \frac{1}{2} \begin{pmatrix}1 & 1\\ 1 & 1\end{pmatrix} = \begin{pmatrix}3/2 & 1/2\\ 1/2 & 3/2\end{pmatrix}.
\]
The eigen-decomposition of \(\Sigma\) is
\[
\Sigma = V\Lambda V^T, \qquad \Lambda = \operatorname{diag}(2, 1), \qquad V = \frac{1}{\sqrt{2}}\begin{pmatrix}1 & 1\\ 1 & -1\end{pmatrix}.
\]
Taking the top eigenpair, we set
\[
U_1 = v_1 \sqrt{\lambda_1} = \frac{1}{\sqrt{2}}\begin{pmatrix}1\\1\end{pmatrix} \cdot \sqrt{2} = \begin{pmatrix}1\\1\end{pmatrix},
\]
and then define the diagonal correction:
\[
D^{(1)} = \operatorname{diag}\left(\Sigma_{11} - (U_1 U_1^T)_{11},\; \Sigma_{22} - (U_1 U_1^T)_{22} \right) = \operatorname{diag}\left( \frac{3}{2} - 1,\; \frac{3}{2} - 1 \right) = \operatorname{diag}\left( \frac{1}{2}, \frac{1}{2} \right).
\]
Thus the decomposition is
\[
D^{(1)} + U_1 U_1^T = \begin{pmatrix}1/2 & 0\\ 0 & 1/2\end{pmatrix} + \begin{pmatrix}1 & 1\\ 1 & 1\end{pmatrix} = \begin{pmatrix}3/2 & 1\\ 1 & 3/2\end{pmatrix},
\]
and the residual is
\[
R = \Sigma - (D^{(1)} + U_1 U_1^T) = \begin{pmatrix}3/2 & 1/2\\ 1/2 & 3/2\end{pmatrix} - \begin{pmatrix}3/2 & 1\\ 1 & 3/2\end{pmatrix} = \begin{pmatrix}0 & -1/2\\ -1/2 & 0\end{pmatrix}.
\]
Since this residual has eigenvalues \(\pm \tfrac{1}{2}\), its operator norm is
\[
\|R\|_2 = \frac{1}{2},
\]
 however the residual is still not zero and does not recover the true structure.

Now, suppose we allow for $T$ number of low-rank factorizations. In that case, we can repeat the process of performing a low-rank factorization of the residual and update the diagonal accordingly. The expectation is that each iteration further reduces the residual norm by capturing additional low-rank structure left in the previous step, progressively improving the approximation. This leads us to Algorithm~\ref{alg:alt_lr_diag_error}.

\begin{algorithm}
\caption{Alternating Low‐Rank then Diagonal (Alt)}
\label{alg:alt_lr_diag_error}
\begin{algorithmic}[1]
\REQUIRE Symmetric matrix $A\in\mathbb{R}^{n\times n}$, target rank $k$, iterations $T$
\ENSURE Approximate decomposition $M = D + U U^\top$

\STATE $D \gets 0_{\,n\times n}$  \COMMENT{Initialize diagonal to zero}
\FOR{$t = 1$ \TO $T$}
  \STATE $R \gets A - D$  \COMMENT{Residual for low‐rank step}
  \STATE Compute eigendecomposition $R = V \Lambda V^\top$, with eigenvalues $\lambda_1 \ge \cdots \ge \lambda_n$
  \STATE $U \gets V_{[:,\,1:k]}\,\sqrt{\diag\bigl(\max(\lambda_1,\dots,\lambda_k,0)\bigr)}$
    \COMMENT{Top-$k$ eigenvectors scaled by $\sqrt{\max(\lambda_i,0)}$}
  \STATE $\text{diagU}_i \gets \sum_{j=1}^k U_{i,j}^2,\quad i=1,\dots,n$  \COMMENT{Compute diagonal of $U U^\top$}
  \STATE $D \gets \diag\!\bigl(\diag(A) - \text{diagU}\bigr)$  \COMMENT{Update diagonal}
\ENDFOR
\STATE $M \gets D + U\,U^\top$
\RETURN $M$
\end{algorithmic}
\end{algorithm}

Figure \ref{fig:errorbound} shows successful recovery using Alt when the model is exactly of LRPD form, $A = LL^\top + D$, with \(L \in \mathbb{R}^{150 \times 5}\) with i.i.d.~standard Gaussian entries and set \(D = \mathrm{diag}(d)\), where each \(d_i\) is drawn uniformly from \([0,10]\). The relative Frobenius error decays to machine precision in 20 iterations.

\begin{theorem}[Sufficient condition for monotone decrease of Alt]
\label{thm:suffcontraction}
Let \(A = D^* + L^*\) with \(D^*=\diag(d^*)\in\R^{n\times n}\) and \(L^*=U^*U^{*T}\in\R^{n\times n}\) of rank \(k\). 
Assume 

\begin{equation} \label{eq:gap_condition}
    \delta=\lambda_k(L^*)>0,
    \qquad
    \|D^*\|_2<\frac{\delta}{2},
\end{equation}

where \(\lambda_k(L^*)\) denotes the \(k\)th eigenvalue of \(L^*\). 
Each alternating step contracts the objective:
\[
E(D,U)
=\bigl\|A - D - UU^T\bigr\|_F^2,
\]
subject to \(D\) diagonal and \(\rank(UU^T)\le k\), giving
\[
E(D_{t-1},U_{t-1})
\ge
E(D_{t-1},U_t)
\ge
E(D_t,U_t),
\]
and moreover, the operator‐norm diagonal errors satisfy
\(\|\Delta_t\|_2\le \|\Delta_{t-1}\|_2\), where \(\Delta_t=D^*-D_t\).
\end{theorem}

\noindent
The next result is needed in the proof to justify that each low‐rank update in Alt is optimal,
as guaranteed by the classical Eckart–Young–Mirsky theorem.

\begin{theorem}[Eckart--Young--Mirsky (EYM).]
\label{thm:eym}
    For any matrix \(R\) with singular values \(\sigma_1 \ge \cdots \ge \sigma_n \ge 0\), the best rank-\(k\) approximation
in both the Frobenius and spectral norms is obtained by truncating the SVD:
\[
\arg\min_{\operatorname{rank}(B)\le k}\|R-B\|_F
\;=\;
\arg\min_{\operatorname{rank}(B)\le k}\|R-B\|_2
\;=\; R_k,
\]
where \(R_k\) keeps the top \(k\) singular values/vectors of \(R\). If \(\sigma_k>\sigma_{k+1}\), the minimizer is unique.
\end{theorem}

\begin{proof}
In this section we will prove Theorem \ref{thm:suffcontraction}.
We first establish monotonic decrease of the alternating minimization. 
At iteration \(t\), define the residual \(R_t = A - D_{t-1}\) and perform the low-rank step
\[
U_t \;\in\; 
\arg\min_{\operatorname{rank}(B)\le k}\;\|R_t - B\|_F^2
\;=\;
\min_U E(D_{t-1},U).
\]
By EYM, the minimizer is the rank-\(k\) truncation of \(R_t\), hence there exists \(U_t\) with
\(U_tU_t^\top = (R_t)_k\) that strictly minimizes \(E(D_{t-1},U)\). Therefore,
\[
E(D_{t-1},U_{t-1})
\;\ge\;
E(D_{t-1},U_t).
\]
Next, with \(U_t\) fixed, write
\[
E(D,U_t)=\|A-D-U_tU_t^\top\|_F^2
=\sum_{i\neq j}\!\bigl(A_{ij}-(U_tU_t^\top)_{ij}\bigr)^2
+\sum_{i=1}^n\bigl(A_{ii}-(U_tU_t^\top)_{ii}-d_i\bigr)^2.
\]
Since the first sum does not depend on \(D\), minimizing over \(D=\diag(d)\) decouples into \(n\) independent
one–dimensional convex quadratics. For each \(i\),
\[
\phi_i(d_i):=\bigl(A_{ii}-(U_tU_t^\top)_{ii}-d_i\bigr)^2,\qquad
\frac{\partial \phi_i}{\partial d_i}=-2\bigl(A_{ii}-(U_tU_t^\top)_{ii}-d_i\bigr).
\]
Setting \(\partial \phi_i/\partial d_i=0\) and using strict convexity gives the unique minimizer
\begin{equation}
\label{eq:diag_update}
d_i^\star
= A_{ii} - (U_tU_t^\top)_{ii},
\qquad\text{equivalently}\qquad
D_t = \diag\!\bigl(A - U_tU_t^\top\bigr).
\end{equation}
Therefore,
\[
E(D_{t-1},U_t)\;\ge\;\min_{D\ \text{diag}}E(D,U_t)\;=\;E(D_t,U_t),
\]
so the alternating step strictly reduces (or preserves) the objective when updating the diagonal.
We now derive the fixed‐rate bound.  Write the residual
\[
R_t = A - D_{t-1}
    = L^* + D^* - D_{t-1}
    = L^* + \Delta_{t-1},
\quad
\Delta_{t-1} = D^* - D_{t-1}.
\]
First recall Weyl’s inequalities (see Theorem 4.3.1 of~\cite{horn2012matrix}) for two symmetric matrices \(X\) and \(Y\) with eigenvalues ordered as
\(\lambda_1(\cdot)\ge\cdots\ge\lambda_n(\cdot)\):
\[
\lambda_i(X+Y)\;\ge\;\lambda_i(X)\;+\;\lambda_n(Y),
\qquad
\lambda_i(X+Y)\;\le\;\lambda_i(X)\;+\;\lambda_1(Y),
\]
for each \(i=1,\dots,n\).  
Applying Weyl’s inequality with \(X=L^*\) and \(Y=\Delta_{t-1}\) gives, for \(i=k\):
\begin{align}
\lambda_k(R_t)
&= \lambda_k\bigl(L^* + \Delta_{t-1}\bigr)
\;\ge\;
\lambda_k(L^*) \;+\;\lambda_n(\Delta_{t-1})
\;\ge\;
\lambda_k(L^*) \;-\;\|\Delta_{t-1}\|_2,
\label{eq:weyl1}\\
\lambda_{k+1}(R_t)
&= \lambda_{k+1}\bigl(L^* + \Delta_{t-1}\bigr)
\;\le\;
\lambda_{k+1}(L^*) \;+\;\lambda_1(\Delta_{t-1})
\;=\;
0 \;+\;\|\Delta_{t-1}\|_2
\;=\;
\|\Delta_{t-1}\|_2,
\label{eq:weyl2}
\end{align}
where we used \(\lambda_{k+1}(L^*)=0\) since \(\rank(L^*)=k\). Therefore the \emph{effective} spectral gap satisfies
\begin{equation} \label{eq:weyl}
    \delta_t
=\lambda_k(R_t) - \lambda_{k+1}(R_t)
\;\ge\;
\bigl[\lambda_k(L^*) - \|\Delta_{t-1}\|_2\bigr]
\;-\;\|\Delta_{t-1}\|_2
\;=\;
\delta - 2\,\|\Delta_{t-1}\|_2.    
\end{equation}
Recall that the Davis–Kahan \(\sin\Theta\) theorem \cite{davis1970rotation} asserts that if \(X\) and \(X+E\) are symmetric with eigenvalues ordered \(\lambda_1\ge\cdots\ge\lambda_n\), and if \(U\) (resp.\ \(\widehat U\)) spans the top-\(k\) eigenspace of \(X\) (resp.\ \(X+E\)), then whenever the spectral gaps
\(\lambda_k(X)-\lambda_{k+1}(X)\) and \(\lambda_k(X+E)-\lambda_{k+1}(X+E)\) are both at least \(\gamma>0\), one has
\[
\bigl\|\,U U^\top \;-\;\widehat U\,\widehat U^\top\bigr\|_2
\;\le\;
\frac{\|E\|_2}{\gamma}.
\]
In our setting, take
\[
X=L^*,\qquad E=\Delta_{t-1},\qquad X+E=R_t=L^*+\Delta_{t-1}.
\]
From the Weyl bound \eqref{eq:weyl},
\begin{equation}
\label{eq:weyl_gap}
\delta_t
\;=\;
\lambda_k(R_t) - \lambda_{k+1}(R_t)
\;\ge\;
\delta - 2\|\Delta_{t-1}\|_2,
\qquad
\delta := \lambda_k(L^*) > 0.
\end{equation}
We set \(D_0 = 0\) and will prove the contraction claim by induction on \(t\).
Now we normalize \(A\) without loss of generality to enforce \(\delta - 2\|\Delta_0\|_2 > 1\).
Recall \(\Delta_0 = D^\ast\) and \(\delta = \lambda_k(L^\ast)\). 
For any \(\alpha > 0\), replace \(A\) by \(\widetilde A = \alpha A\).
Then \(L^\ast \mapsto \widetilde L^\ast = \alpha L^\ast\) and \(D^\ast \mapsto \widetilde D^\ast = \alpha D^\ast\), so
\[
\widetilde\delta = \lambda_k(\widetilde L^\ast) = \alpha\delta,
\qquad
\|\widetilde\Delta_0\|_2 = \|\widetilde D^\ast\|_2 = \alpha\|D^\ast\|_2,
\]
and hence
\[
\widetilde\delta - 2\|\widetilde\Delta_0\|_2
= \alpha\bigl(\delta - 2\|D^\ast\|_2\bigr).
\]
By assumption we have \(\delta>2\|D^\ast\|_2\); choosing
\(\alpha>1/(\delta-2\|D^\ast\|_2)\) makes \(\widetilde\delta-2\|\widetilde\Delta_0\|_2>1\).
The alternating updates scale linearly with \(A\), and the relative error
\(E_t=\|A-D_t-U_tU_t^\top\|_F/\|A\|_F\) is invariant under positive scalings of \(A\).

\emph{Base step ($t=1$).} Since $\Delta_0=D^*$ and $\|D^*\|_2<\delta/2$, we have
\[
\delta_1
\;\ge\;
\stackrel{\eqref{eq:weyl_gap}}{=}
\;\delta - 2\|\Delta_0\|_2
\;=\;
\delta - 2\|D^*\|_2
\;>\;
1.
\]
Thus Davis–Kahan with $\gamma=\delta_1$ gives
\begin{equation}
\label{eq:davis_kahan}
\|U_1U_1^\top - U^*U^{*\top}\|_2
\;\le\;
\frac{\|\Delta_0\|_2}{\delta_1}
\;\le\;
\frac{\|\Delta_0\|_2}{\delta - 2\|\Delta_0\|_2}.
\end{equation}
and the Rayleigh bound yields
\[
\|\Delta_1\|_2
\stackrel{\eqref{eq:diag_update}}{=}
\|\diag(U_1U_1^\top - L^*)\|_\infty
\;\le\;
\|U_1U_1^\top - L^*\|_2
\;\stackrel{\eqref{eq:davis_kahan}}{\le}\;
\frac{\|\Delta_0\|_2}{\delta - 2\|\Delta_0\|_2}
\;\le\;
\|\Delta_0\|_2.
\]

\emph{Induction step.} Suppose $\|\Delta_{t-1}\|_2\le\|\Delta_{t-2}\|_2\le\cdots\le\|\Delta_0\|_2$.
Then
\[
\delta_t
\;\stackrel{\eqref{eq:weyl_gap}}{\ge}\;
\delta - 2\|\Delta_{t-1}\|_2
\;\ge\;
\delta - 2\|\Delta_0\|_2
\;>\;
1.
\]
Repeating the sequence with \(U_t\) as the EYM minimizer for \(R_t\), 
the Davis–Kahan and Rayleigh bounds give
\[
\|\Delta_t\|_2
\;\le\;
\|U_tU_t^\top - L^*\|_2
\;\stackrel{\eqref{eq:davis_kahan}}{\le}\;
\frac{\|\Delta_{t-1}\|_2}{\delta_t}
\;\stackrel{\eqref{eq:weyl_gap}}{\le}\;
\frac{\|\Delta_{t-1}\|_2}{\delta - 2\|\Delta_{t-1}\|_2}
\;\le\;
\|\Delta_{t-1}\|_2.
\]
Hence, by induction,
\[
\|\Delta_t\|_2\ \le\ \|\Delta_{t-1}\|_2\ \le\ \cdots\ \le\ \|\Delta_0\|_2=\|D^*\|_2,
\]
i.e., the diagonal operator–norm error is nonincreasing, and in particular $\|\Delta_{t-1}\|_2\le\|D^*\|_2$ holds.
Finally, since the diagonal update is the exact minimizer of
\(E(D,U_t)=\|A-D-U_tU_t^\top\|_F^2\) over diagonal \(D\), we also have
\(E(D_t,U_t)\le E(D_{t-1},U_t)\), aligning the decrease of \(E\) with the monotonic decay of \(\|\Delta_t\|_2\).
\end{proof}

\begin{rmk}
      Unlike standard gradient-descent approaches, which can suffer from degeneracy because any factorization $U$ satisfying $U U^\top$ can be rotated without changing the product, given distinct eigenvalues our spectral procedure directly computes the top-$k$ eigenpairs of the residual matrix. Consequently, the low-rank component is uniquely determined up to the ordering of eigenvectors:
\[
U U^\top = \sum_{i=1}^k \lambda_i u_i u_i^\top,
\]
where $\lambda_i$ and $u_i$ are the top $k$ eigenvalues and corresponding eigenvectors of the residual. This eliminates the rotational degeneracy present in gradient-based factorization methods.   
\end{rmk}

\begin{figure}[H]
    \centering
    \includegraphics[width=0.5\textwidth]{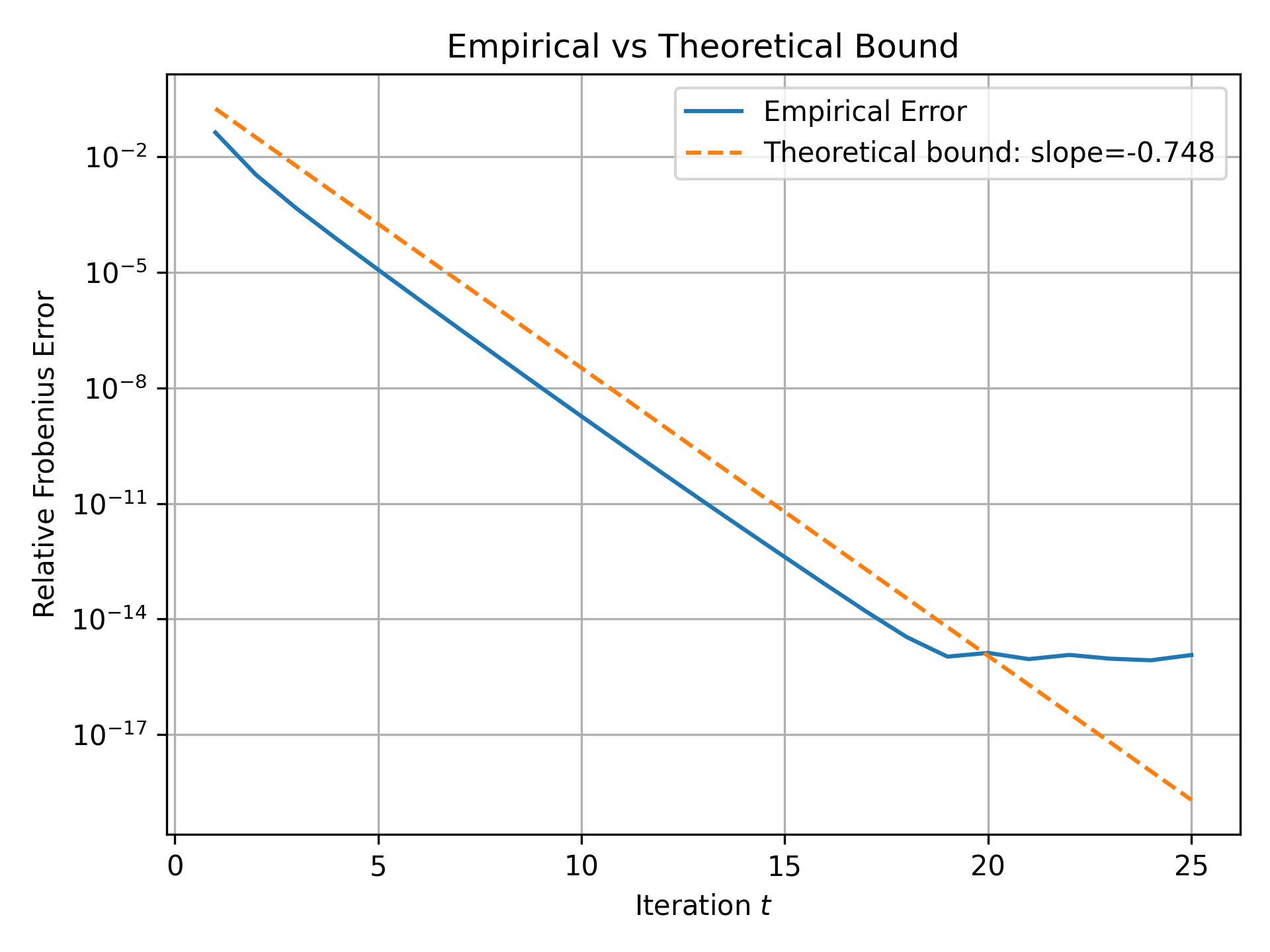}
    \caption{Empirical convergences of Algorithm \ref{alg:alt_lr_diag_error} and theoratical bound from Theorem \ref{thm:suffcontraction}, where $n=150$ and $f=k=5$.}
    \label{fig:errorbound}
\end{figure}

Theorem~\ref{thm:suffcontraction} provides sufficient conditions guaranteeing a monotone decrease in the error of the alternating low-rank–diagonal iteration, but these conditions \eqref{eq:gap_condition} depend explicitly on the presence of an eigenvalue gap.  Since
the behavior of the iteration is governed primarily by the geometry of the leading eigenspace,
encoded in its orthogonal projector, one expects a more intrinsic and gap independent description.
Theorem~\ref{thm:gapfree} confirms this intuition by showing that the map is always locally nonexpansive, and that
it becomes strictly contractive whenever the top k subspace is not aligned with the coordinate axes.

\begin{theorem}[Gap-free local contraction]\label{thm:gapfree}
Let $A\in\mathbb{R}^{n\times n}$ be symmetric with distinct eigenvalues and suppose $A=D^*+L^*$ with $D^*$ diagonal and $L^*$ symmetric of rank $k$.
Consider the alternating updates
\begin{equation} \label{eq:gapfreeiter}
    L_{t+1}=T_k(A-D_t),\qquad
    D_{t+1}=\operatorname{diag}(A-L_{t+1}),
\end{equation}
where $T_k(\cdot)$ keeps the top-$k$ eigenpairs and zeros the rest.
Let $P:=U^*{U^*}^\top$ be the projector onto the top-$k$ eigenspace of $L^*$ and set $Q:=I-P$.
Then $(D^*,L^*)$ is a fixed point of \eqref{eq:gapfreeiter}. Writing the errors $e_D^t:=D_t-D^*$ and $e_L^t:=L_t-L^*$, the one-step error map admits the first-order expansion
\[
\begin{bmatrix} e_D^{t+1} \\[2pt] e_L^{t+1} \end{bmatrix}
=
\underbrace{\begin{bmatrix}
\mathcal{J}_D(e_D) & 0\\[2pt]
-\Lambda & 0
\end{bmatrix}}_{\displaystyle \mathcal{J}}
\begin{bmatrix} e_D^{t} \\[2pt] e_L^{t} \end{bmatrix}
\;+\;o\!\big(\|(e_D^{t},e_L^{t})\|\big),
\]
where $\Lambda:\mathbb{R}^{n\times n}\to\mathbb{R}^{n\times n}$ is the linear map
\[
\Lambda(E)\;=\;PE+EP-PEP\;=\;E-QEQ,
\qquad E\in\mathbb{R}^{n\times n},
\]
with the reduced (diagonal) Jacobian
\[
\mathcal{J}_D(e_D)=\operatorname{diag}\big(\Lambda(e_D)\big)
=\operatorname{diag}\big(e_D-Qe_DQ\big).
\]
Then,
\[
\|\mathcal{J}_D\|_{\infty\to\infty}\le 1.
\]
Equality $\|\mathcal{J}_D\|_{\infty\to\infty}=1$ holds iff $P$ is a coordinate (axis-aligned) projector
($Q$ diagonal with $Q_{ii}\in\{0,1\}$), in which case
$(\mathcal{J}_D e_D)_i$ for $i\in\mathrm{supp}(P)$ and $0$ otherwise; otherwise
$\|\mathcal{J}_D\|_{\infty\to\infty}<1$ and the reduced $D$-map is a strict contraction.

\medskip
The full Jacobian $\mathcal{J}$ is block–lower–triangular, hence
\[
\rho(\mathcal{J})=\max\{\rho(\mathcal{J}_D),\rho(0)\}=\rho(\mathcal{J}_D).
\]
Therefore, if $\|\mathcal{J}_D\|_{\infty\to\infty}<1$ then $\rho(\mathcal{J})<1$.
\end{theorem}

\begin{proof}
\emph{Fixed point.}
Since $A=D^*+L^*$ and $T_k(L^*)=L^*$, we have $L^*=T_k(A-D^*)$ and $D^*=\operatorname{diag}(A-L^*)$, hence $(D^*,L^*)$ is a fixed point.

\smallskip
\emph{First-order Taylor expansion for the $L$-update.}
Set $e_D:=D-D^*$ and write
\[
A-D \;=\; (D^*+L^*)-(D^*+e_D) \;=\; L^* - e_D.
\]
Define the perturbation $E:=-e_D$. Since $T_k$ is a spectral (matrix) function that keeps the top-$k$ spectral subspace, its Fr\'echet derivative at $L^*$ exists and the Taylor expansion reads
\[
T_k(L^*+E)
\;=\;T_k(L^*)\;+\;\mathbf{D}T_k[L^*]\{E\}\;+\;o(\|E\|).
\]
Using $T_k(L^*)=L^*$ and $E=-e_D$ gives
\begin{equation} \label{e:taylorLstar}
    L^+ - L^* \;=\; \mathbf{D}T_k[L^*]\{-e_D\} \;+\; o(\|e_D\|).
\end{equation}
We now compute $\mathbf{D}T_k[L^*]$ explicitly.  
Let $L^*=U\Lambda U^\top$ be the eigendecomposition of $L^*$, where $U$ is orthogonal and
$\Lambda=\operatorname{diag}(\lambda_1,\ldots,\lambda_n)$.  
By the Daleckii--Kre\u{\i}n formula (Corollary~3.12 of~\cite{higham2008functions}),
\begin{equation}
\label{eq:dk-collapse}
\mathbf{D}T_k[L^*](E)
=
U\bigl(f^{[1]}(\Lambda)\circ(U^\top E U)\bigr)U^\top,
\end{equation}
where $f^{[1]}(\Lambda)$ is the first–order divided–difference matrix of the scalar function $f$
associated with $T_k$, evaluated on the eigenvalues in $\Lambda$, and $\circ$ denotes the Hadamard
product.  Its entries are
\[
f^{[1]}(\lambda_i,\lambda_j)
=
\begin{cases}
\dfrac{f(\lambda_i)-f(\lambda_j)}{\lambda_i-\lambda_j}, & i\neq j,\\[6pt]
f'(\lambda_i), & i=j.
\end{cases}
\]

Corollary~3.12 of~\cite{higham2008functions} requires the underlying scalar function to be
$C^{2n-1}$ on a domain containing the spectrum of $L^*$, a condition not satisfied by the
piecewise--$C^1$ truncation map that equals the identity on the top spectral cluster and vanishes on
the bottom cluster.  To meet the differentiability assumptions while preserving the action of
$T_k$ on $\sigma(L^*)$, we introduce a smooth mollified surrogate.  Since the eigenvalues of $L^*$
admit a strict gap $\lambda_k>\lambda_{k+1}$, choose numbers $\lambda_{k+1}<\alpha<\beta<\lambda_k$
and let $\varphi\in C^\infty(\mathbb{R})$ be any smooth cutoff satisfying
$\varphi(t)=0$ for $t\le 0$, $\varphi(t)=1$ for $t\ge 1$, and $0\le \varphi(t)\le 1$.
Define
\[
\chi(\lambda)=\varphi\!\left(\frac{\lambda-\alpha}{\beta-\alpha}\right),
\qquad
f(\lambda)=\lambda\,\chi(\lambda).
\]
Then $f\in C^\infty(\mathbb{R})$, and the choice of $\alpha,\beta$ ensures that
$f(\lambda_i)=\lambda_i$ for all $i\le k$ and $f(\lambda_i)=0$ for all $i>k$, so that
$T_k(L^*)=U f(\Lambda) U^\top$.  In this way we obtain a smooth spectral representative of the
hard truncation on $\sigma(L^*)$, making the Daleckii--Kre\u{\i}n formula applicable while leaving
the action of $T_k$ on $L^*$ unchanged.  We continue to denote by $P=U^*{U^*}^\top$ the spectral
projector onto the top--$k$ eigenspace and set $Q:=I-P$. Then
\[
f[\lambda_i,\lambda_j]=
\begin{cases}
1,& i,j\le k\quad(PP),\\
1,& i\le k<j\ \text{or}\ j\le k<i\quad(PQ/QP),\\
0,& i,j>k\quad(QQ),
\end{cases}
\]
so
\[
f^{[1]}(\Lambda)=
\begin{bmatrix}
\mathbf{1}_{PP} & \mathbf{1}_{PQ}\\[2pt]
\mathbf{1}_{QP} & \mathbf{0}_{QQ}
\end{bmatrix}.
\]
Writing \(\widehat E:=U^\top E U=\begin{bmatrix}E_{PP}&E_{PQ}\\[2pt]E_{QP}&E_{QQ}\end{bmatrix}\), we obtain
\[
\mathbf{D}T_k[L^*]\{E\}
=U\!\begin{bmatrix}E_{PP}&E_{PQ}\\[2pt]E_{QP}&0\end{bmatrix}\!U^\top
\]
With \(U=[\,U_1\ \ U_2\,]\), \(P:=U_1U_1^\top\), \(Q:=U_2U_2^\top\), and
\(
E_{PP}=U_1^\top E U_1,\;
E_{PQ}=U_1^\top E U_2,\;
E_{QP}=U_2^\top E U_1,
\)
we compute
\[
\begin{aligned}
U\!\begin{bmatrix}E_{PP}&E_{PQ}\\[2pt]E_{QP}&0\end{bmatrix}
&=\big[\,U_1\ \ U_2\,\big]
   \begin{bmatrix}E_{PP}&E_{PQ}\\[2pt]E_{QP}&0\end{bmatrix}
 =\big[\,U_1E_{PP}+U_2E_{QP}\ \ \, U_1E_{PQ}\,\big],\\[4pt]
U\!\begin{bmatrix}E_{PP}&E_{PQ}\\[2pt]E_{QP}&0\end{bmatrix}\!U^\top
&=\big[\,U_1E_{PP}+U_2E_{QP}\ \ \, U_1E_{PQ}\,\big]
   \begin{bmatrix}U_1^\top\\[2pt]U_2^\top\end{bmatrix}\\[4pt]
&=(U_1E_{PP})U_1^\top \;+\; (U_2E_{QP})U_1^\top \;+\; (U_1E_{PQ})U_2^\top\\[2pt]
&=U_1(U_1^\top E U_1)U_1^\top \;+\; U_2(U_2^\top E U_1)U_1^\top \;+\; U_1(U_1^\top E U_2)U_2^\top\\[2pt]
&=(U_1U_1^\top)E(U_1U_1^\top) \;+\; (U_2U_2^\top)E(U_1U_1^\top) \;+\; (U_1U_1^\top)E(U_2U_2^\top)\\[2pt]
&=\,PEP \;+\; QEP \;+\; PEQ\,.
\end{aligned}
\]

Because $Q = I-P$, we expand each term of 
\(
PEP + PEQ + QEP
\)
in terms of $P$, $E$, and $I$.
First,
\[
PEQ = P E (I-P) = PE - PEP.
\]
Second,
\[
QEP = (I-P) E P = EP - PEP.
\]
Substituting these expressions into the sum gives
\[
PEP + PEQ + QEP
= PEP + (PE - PEP) + (EP - PEP).
\]
Collecting like terms,
\[
PEP + PE - PEP + EP - PEP
= PE + EP - PEP,
\]
since the $PEP$ terms cancel except for one with negative sign.
Thus we obtain
\[
\Lambda(E)
:=\mathbf{D}T_k[L^*]\{E\}
= PE + EP - PEP.
\]
To express the same object using $Q = I-P$, observe that
\[
QEQ = (I-P)E(I-P)
= E - EP - PE + PEP.
\]
Rearranging gives
\[
E - QEQ = E - \bigl(E - EP - PE + PEP\bigr)
= PE + EP - PEP.
\]
Hence the two expressions coincide, and we obtain the equivalent identities
\begin{equation} \label{e:derivoperator}
    \Lambda(E)
= \mathbf{D}T_k[L^*]\{E\}
= PE + EP - PEP
= E - QEQ.    
\end{equation}
Therefore, by \eqref{e:taylorLstar} and \eqref{e:derivoperator}, 
\[
e_L^+ \;:=\; L^+ - L^*
\;=\; -\,\Lambda(e_D) \;+\; o(\|e_D\|),
\]
which shows that, to first order, the new low-rank error depends \emph{only} on $e_D$ (and not on $e_L$).

\smallskip
\emph{First-order Taylor expansion for the $D$-update.}
Using $D^+=\operatorname{diag}(A-L^+)$ and $A=D^*+L^*$,
\[
e_D^+ \;:=\; D^+ - D^*
\;=\; \operatorname{diag}\big(A-L^+\big) - \operatorname{diag}\big(A-L^*\big)
\;=\; \operatorname{diag}\big(L^*-L^+\big)
\;=\; -\,\operatorname{diag}(e_L^+).
\]
Substitute the first-order expression for $e_L^+$:
\[
e_D^+ \;=\; \operatorname{diag}\!\big(\Lambda(e_D)\big) \;+\; o(\|e_D\|).
\]
Thus the Jacobian of the reduced $D$-map is
\[
\mathcal{J}_D:\ e_D \ \mapsto\ \operatorname{diag}\!\big(\Lambda(e_D)\big),
\qquad\text{i.e.}\qquad
\mathcal{J}_D(e_D)=\operatorname{diag}\!\big(E - QEQ\big)
\quad\text{with }E=e_D.
\]

\smallskip
Recall $\Lambda(E)=PE+EP-PEP=E-QEQ$ and $\mathcal{J}_D(e_D)=\operatorname{diag}(\Lambda(e_D))$.
Let $e_D=\mathrm{diag}(\varepsilon_1,\ldots,\varepsilon_n)$ and $Q:=I-P$.
Then, componentwise,
\begin{equation} \label{eq:jacobian_comp}
    (\mathcal{J}_D e_D)_i
=\big[\Lambda(e_D)\big]_{ii}
=\big[e_D - Q e_D Q\big]_{ii}
=\varepsilon_i-\sum_{j=1}^n Q_{ij}^2\,\varepsilon_j.
\end{equation}

Separating the $j=i$ term gives
\[
(\mathcal{J}_D e_D)_i
=(1-Q_{ii})\,\varepsilon_i \;-\; \sum_{j\neq i} Q_{ij}^2\,\varepsilon_j.
\]
Taking absolute values and using $\|e_D\|_\infty:=\max_j|\varepsilon_j|$ we get,
\[
\big|(\mathcal{J}_D e_D)_i\big|
\le
\Big[(1-Q_{ii})+\sum_{j\neq i}Q_{ij}^2\Big]\;\|e_D\|_\infty.
\]
Since $Q$ is an orthogonal projector ($Q^\top=Q$, $Q^2=Q$), its $i$th row satisfies the row–square identity
\[
\sum_{j=1}^n Q_{ij}^2=(Q^2)_{ii}=Q_{ii}
\quad\Longrightarrow\quad
\sum_{j\neq i}Q_{ij}^2=Q_{ii}-Q_{ii}^2.
\]
Therefore,
\[
\big|(\mathcal{J}_D e_D)_i\big|
\le \big[(1-Q_{ii})+(Q_{ii}-Q_{ii}^2)\big]\|e_D\|_\infty
=(1-Q_{ii}^2)\,\|e_D\|_\infty
\le \|e_D\|_\infty,
\]
because $0\le Q_{ii}\le 1\Rightarrow 0\le Q_{ii}^2\le 1$. Maximizing over $i$ yields the dimension-free bound
\[
\quad \|\mathcal{J}_D\|_{\infty\to\infty}\le 1.\quad
\]

\smallskip
If $P$ is a coordinate (axis-aligned) projector, then $Q$ is diagonal and $Q_{ij}=0$ for $i\neq j$.
From \eqref{eq:jacobian_comp} we get $(\mathcal{J}_D e_D)_i=\varepsilon_i-Q_{ii}^2\varepsilon_i$, so
$(\mathcal{J}_D e_D)_i=\varepsilon_i$ for $i\in \mathrm{supp}(P)$ and $0$ otherwise; hence
$\|\mathcal{J}_D\|_{\infty\to\infty}=1$ (nonexpansive). Whenever $P$ is not axis-aligned, some off-diagonal
$Q_{ij}\neq 0$ forces $Q_{ii}\in(0,1)$ and $1-Q_{ii}^2<1$, giving strict contraction.

\smallskip

\emph{Axis-aligned equality case.}
If $Q_{ii}=1$, then $(\mathcal{J}_D e_D)_i=\varepsilon_i - Q_{ii}^2\varepsilon_i=0$, so the $i$th
coordinate is fully suppressed.  If $Q_{ii}=0$, then $(\mathcal{J}_D e_D)_i=\varepsilon_i$, i.e.,
that coordinate is preserved exactly.  For a coordinate projector
$P=\mathrm{diag}(\mathbf{1}_S)$, $Q$ is diagonal with $Q_{ii}\in\{0,1\}$, giving
\[
(\mathcal{J}_D e_D)_i =
\begin{cases}
\varepsilon_i, & i\in S,\\[3pt]
0, & i\notin S,
\end{cases}
\qquad
\|\mathcal{J}_D\|_{\infty\to\infty}=1.
\]

\end{proof}

\begin{rmk}
    If $P$ is not axis-aligned, some $Q_{ij}\neq 0$ implies $0<Q_{ii}<1$, hence
    $1-Q_{ii}^2<1$ and we obtain a strict contraction.  A simple example illustrates why axis
    alignment is the only noncontractive case.  Suppose the true low-rank component $L^*$ is itself
    diagonal.  Then $P$ is a coordinate projector and $Q=I-P$ is also diagonal, which forces $QEQ$ to
    act entrywise.  In this situation the model $A=D^*+L^*$ is indistinguishable from a single diagonal
    matrix, because two diagonal matrices cannot be separated by any low-rank projection.  For any
    diagonal $E$ we have $PEP=0$ whenever the support of $E$ lies in the $Q$-block, and the same holds
    in the reverse direction.  The problem is therefore degenerate, since there is no geometric mixing
    between coordinates, and the Jacobian reduces to $(\mathcal{J}_D e_D)_i=\varepsilon_i$ on the
    active diagonal coordinates and zero elsewhere.  This situation is exactly the axis-aligned case
    where $\|\mathcal{J}_D\|_{\infty\to\infty}=1$ and only nonexpansive (but not strict) contraction
    can hold.
\end{rmk}


To place our method in context, we compare its performance with a classical
majorization--minimization (MM) strategy for fixed--rank approximation with a diagonal
correction.  We now detail the MM construction used in this comparison.  With $D$ fixed and
$B:=A-D$, define
\[
g(U)=\|B-UU^\top\|_F^2=\operatorname{tr}\!\big((B-UU^\top)^2\big).
\]
Using the identities
\[
\mathrm{d}\,\operatorname{tr}(U^\top B U)=2\,B U:\mathrm{d}U
\quad\text{(for symmetric $B$)}, 
\qquad
\mathrm{d}\,\operatorname{tr}(UU^\top UU^\top)=4\,U(U^\top U):\mathrm{d}U,
\]
we obtain the gradient
\[
\nabla g(U)=4\big(U(U^\top U)-BU\big).
\]
Let $L>0$ be any Lipschitz constant of $\nabla g$.  By the Descent Lemma
\cite{sun2016majorization}, the quadratic surrogate
\[
q(U\mid U_t)
=
g(U_t)
+\langle\nabla g(U_t),\,U-U_t\rangle
+\tfrac{L}{2}\|U-U_t\|_F^2
\]
majorizes $g$ at $U_t$.  The MM update minimizes this surrogate, yielding the
proximal--gradient step
\[
U_{t+1}=U_t-\tfrac{1}{L}\,\nabla g(U_t),
\]
with backtracking (increasing $L$) until the majorization inequality holds.
In each outer iteration we then take the exact diagonal update
$ D_{ii}\gets A_{ii}-\sum_{j=1}^k U_{ij}^2.$
We initialize $U_0$ with the rank-$k$ eigendecomposition of $A$.

Figure~\ref{fig:comparisontoMM} shows that our alternating low rank plus diagonal method (Alt)
consistently outperforms the Majorization Minimization (MM) baseline.

Finally, since the eigendecomposition dominates the flop count of Alt, we replace it with a Nyström sketch \cite{tropp2017fixed} that uses a Gaussian sketch matrix $\Omega \in \mathbb{R}^{n\times k}$, where $k$ is the target rank. This reduces the low-rank step from $\Theta(n^3)$ to the cost of forming the sketch $Y = R\Omega$ and working in the reduced space, which is $\Theta(n^2k)$ for the multiplication with $R$, plus lower-order terms $\Theta(k^3)$, for the eigendecomposition of the small sketch matrix. The diagonal update remains $\Theta(nk)$. Thus, the overall cost per iteration of the Nyström variant is $\Theta(n^2k)$, dominated by the matrix–sketch multiply.

By comparison, the MM method requires one full eigendecomposition to initilize $U_0$ at $\Theta(n^3)$, then in each inner step a multiplication $BU$ at $\Theta(n^2k)$, plus $U^\top U$ and $U(U^\top U)$ at $\Theta(nk^2)$, for a total of $\Theta(n^3 + n^2k + nk^2)$. The diagonal update is again $\Theta(nk)$. Therefore, the Nyström-based Alt is much cheaper than MM per iteration, but the former directly recovers an approximate spectral factorization while the latter progresses via gradient iterations. Because the Nyström sketch removes the $\Theta(n^3)$ bottleneck of the full eigendecomposition, this becomes our new algorithm; we will analyze its performance in detail in the next section.

For the numerical experiment in Figure~\ref{fig:comparisontoMM}, we generate a synthetic covariance matrix with controlled low-rank structure and diagonal noise. We first draw a random Gaussian matrix $U_0 \in \mathbb{R}^{n\times k}$ and normalize its columns, then assign a decaying spectrum $s = (3, 3-\tfrac{2}{k}, \ldots, 1)$ to weight the columns, producing the low-rank factor $(U_0 s)(U_0 s)^\top$. To this structured component we add a diagonal matrix with heterogeneous positive entries sampled uniformly from $[0.2,1.2]$, modeling idiosyncratic variances across coordinates. Finally, to avoid a perfectly clean model, we inject a small symmetric Gaussian noise matrix at a very high signal-to-noise ratio, ensuring that the low-rank-plus-diagonal structure remains dominant while still presenting a realistic perturbation. This construction yields a well-conditioned test problem that mimics the spectral profile of empirical covariance matrices and allows a controlled comparison of the alternating method, its stochastic sketch variant, and the MM baseline.

\begin{figure}[H]
    \centering
    \includegraphics[width=0.5\textwidth]{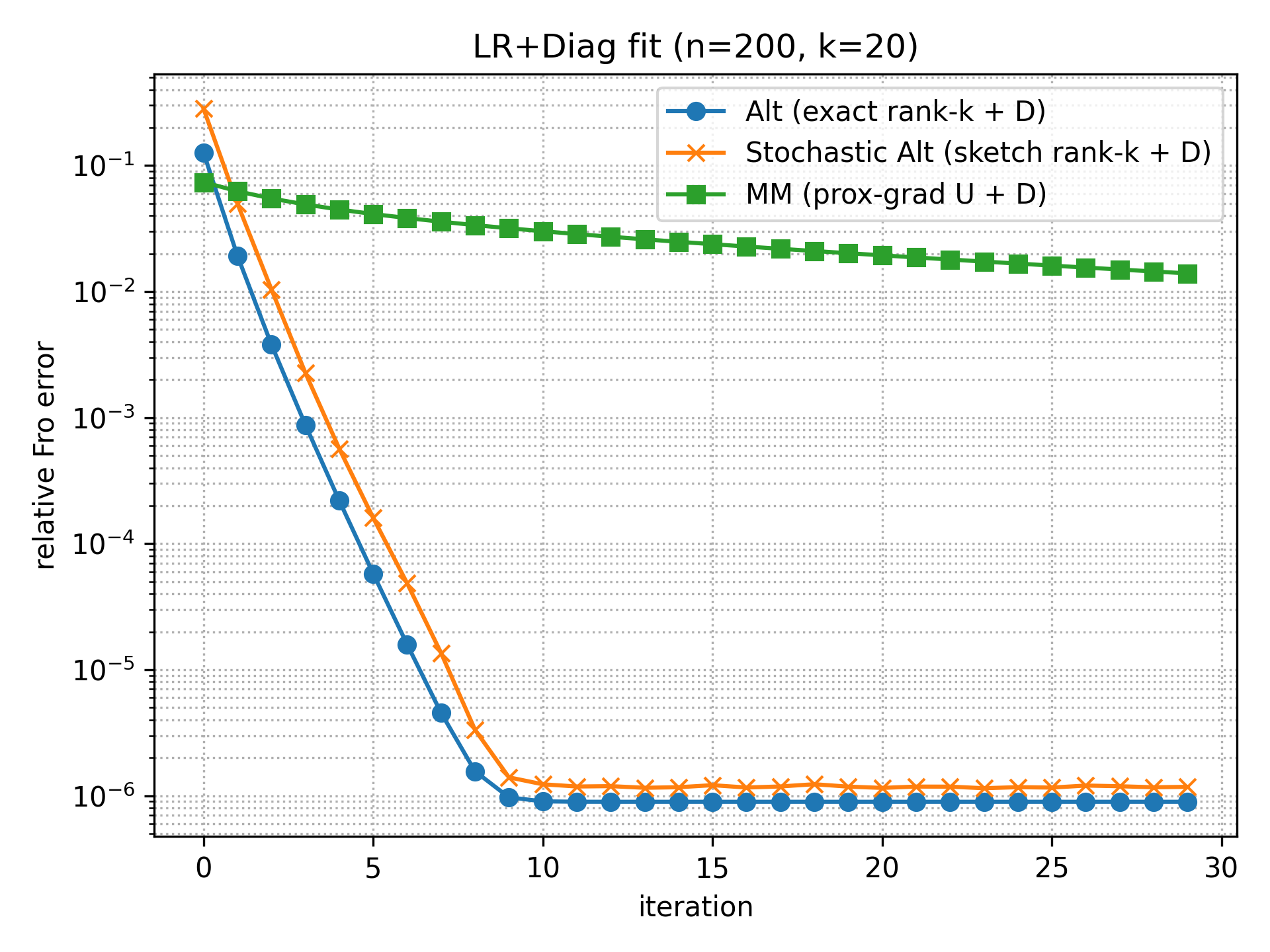}
    \caption{Relative Frobenius error versus iteration for low-rank plus diagonal approximation ($k=20$, $n=200$). The plot compares the alternating low-rank + diagonal method (Alt) against the MM (proximal-gradient) approach \cite{sun2016majorization}}
    \label{fig:comparisontoMM}
\end{figure}

We evaluate the Alt and MM LR+Diag algorithms under varying levels of additive Gaussian noise. Let $A_0 \in \mathbb{R}^{n \times n}$ denote the ground-truth low-rank plus diagonal matrix, constructed as
\[
A_0 = \mathrm{diag}(d_1, \dots, d_n) + U S U^\top,
\]
where $U \in \mathbb{R}^{n \times k}$ has orthonormal columns, and $S = \mathrm{diag}(s_1, \dots, s_k)$ contains the singular values. We add symmetric Gaussian noise $N = (N + N^\top)/2$ and scale it to achieve a target signal-to-noise ratio (SNR) in decibels:
\[
A = A_0 + \alpha N, \quad 
\mathrm{SNR}_{\mathrm{dB}} = 10 \log_{10} \frac{\|A_0\|_F^2}{\|\alpha N\|_F^2}, \quad 
\alpha = \sqrt{\frac{\|A_0\|_F^2}{\|N\|_F^2 10^{\mathrm{SNR}_{\mathrm{dB}}/10}}}.
\]
Higher $\mathrm{SNR}_{\mathrm{dB}}$ corresponds to smaller noise perturbations, e.g., $120$~dB is nearly noiseless, $60$~dB represents moderate noise, and $20$~dB is highly noisy. Figure~\ref{fig:noiselevels} reports the relative Frobenius error $\|M - A\|_F / \|A\|_F$ over iterations for both Alt and MM under these noise conditions.

\begin{figure}[H]
    \centering
    \includegraphics[width=0.5\textwidth]{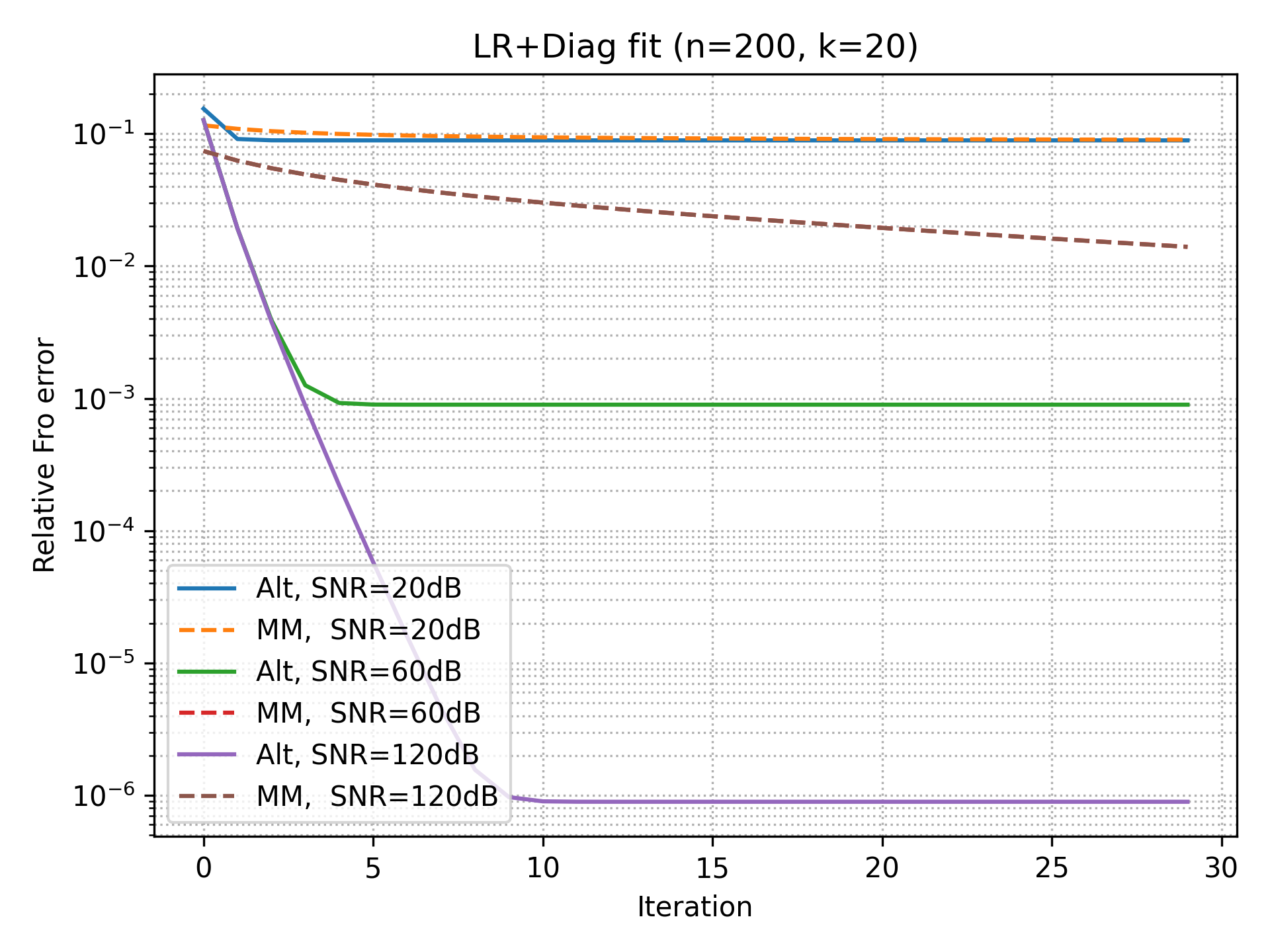}
    \caption{Relative Frobenius error versus iteration for Alt (solid lines) and MM (dashed lines) under different SNR levels. Higher dB corresponds to lower noise.}
    \label{fig:noiselevels}
\end{figure}

\section{Randomized LRPD decomposition}

We have previously seen that the Nystr\"om sketch dramatically reduces the computational cost while
producing accuracy comparable to the full low rank plus diagonal decomposition.  In this section we
analyze these benefits.  In many applications we have access only to
matrix–vector products with the matrix of interest, rather than its explicit entries.  This situation
arises when the matrix is too large to store in memory and we seek an accurate approximation using
significantly fewer matrix–vector products than the ambient dimension.  To address this regime, we
propose a randomized algorithm that combines two complementary tools: stochastic estimation of the
diagonal~\cite{baston2022stochastic} and randomized Nystr\"om approximation for low rank
sketching~\cite{frangella2023randomized, tropp2017fixed, szlam2014implementation}.

A \emph{fixed-rank Nyström method} constructs a rank-$r$ approximation of a symmetric positive semidefinite matrix $A \in \mathbb{R}^{n \times n}$ by first drawing a sketching matrix $\Omega \in \mathbb{R}^{n \times k}$ (typically Gaussian or orthonormal with $k \ge r$), forming the sample matrix $Y = A \Omega$, and computing the small $k \times k$ Gram matrix $C = \Omega^\top Y$. A rank-$r$ approximation is then given by  
\[
\widehat{A}_r \;=\; Y\,C_r^{\dagger}\,Y^\top,
\]  
where $C_r$ is the best rank-$r$ approximation of $C$ (via truncated eigenvalue decomposition). This method achieves a relative error bound that depends on the spectral tail of $A$ and the sketch size $k$, with the sharp constants provided in~\cite{tropp2017fixed}.  

Theorem~\ref{thm:stochbound} follows by integrating the Diag++ estimator~\cite{baston2022stochastic}, which achieves a relative error~$\varepsilon$ using only $O(1/\varepsilon)$ matrix–vector products, together with the relative-error guarantees for the fixed-rank randomized Nyström method~\cite{tropp2017fixed}.

\begin{theorem}[Error bound for a single iterate of Stochastic Alt] \label{thm:stochbound}
Let \(A\in\R^{n\times n}\) be psd, and fix a target rank \(r<k\le n\).  At iteration \(t\) define the residual
\[
R_t \;=\; A - D_{t-1}.
\]
Perform two randomized steps:
\begin{enumerate}
  \item[\(\bullet\)] A fixed‐rank Nyström approximation of \(R_t\) of rank \(r\), using sketch size \(k\) (equivalently \(k\) matrix–vector products with \(A\)), producing \(U_tU_t^T\).
  \item[\(\bullet\)] A Diag\(++\) estimator with \(s\) Rademacher queries to approximate \(\diag(R_t)\), producing the diagonal matrix \(D_t\).
\end{enumerate}
Let \(b=\max\{k,s\}\) denote the unified per–iteration matrix–vector budget.  
Then with probability at least \(1-2\delta\),
\[
\bigl\|A - (D_t + U_t U_t^T)\bigr\|_{\infty}
\;\le\;
\underbrace{\|R_t - (R_t)_r\|_{\infty}
+\frac{r}{\,k-r-\alpha\,}\,\|R_t - (R_t)_r\|_{1}}_{\displaystyle E_{\rm lr}}
\;+\;
\underbrace{\varepsilon\,\|\diag(R_t)\|_{2}}_{\displaystyle E_{\rm diag}},
\]
where \(\alpha=1\) (real field) or \(0\) (complex) and \((R_t)_r\) is the best rank-\(r\) approximation of \(R_t\). \({\displaystyle E_{\rm lr}}\) is from Theorem 4.1 \cite{tropp2017fixed} and \({\displaystyle E_{\rm diag}}\) is from section 9.2 Diag++ \cite{baston2022stochastic}.

In particular, Theorem 4.1 \cite{tropp2017fixed} implies that for each \(\varepsilon>0\),
\[
k = (1+\varepsilon^{-1})\,r + \alpha
\;\Longrightarrow\;
\mathbb{E}\,\|R_t - \widehat R_{t,r}\|_{1} \le (1+\varepsilon)\,\|R_t - (R_t)_r\|_{1}.
\]
Likewise, to ensure the diagonal error satisfies
\(E_{\rm diag}\le \varepsilon\,\|\diag(R_t)\|_{2}\) with probability \(1-\delta\), it suffices to take
\[
s \;>\;
\frac{4\,\tr(R_t)}{\|\diag(R_t)\|_{2}}
\,\frac{\sqrt{\ln(2n/\delta)}}{\varepsilon}
\;+\;c\,\ln\frac1\delta
\]
for some constant \(c\).  Hence, allocating a per–iteration budget \(b=\max\{k,s\}\) yields an explicit bound on the overall max–norm error of the stochastic Alt iterate.
\end{theorem}

\begin{proof}
We split the error into off–diagonal and diagonal parts.

\textbf{Low‐rank (off‐diagonal) error (Theorem 4.1 \cite{tropp2017fixed}).}  
Applying the fixed–rank Nyström approximation theorem to the psd matrix \(R_t\) shows that the rank-\(r\) Nyström approximation \(U_tU_t^T\) satisfies
\[
\|R_t - U_tU_t^T\|_{\infty}
\;\le\;
\|R_t - (R_t)_r\|_{\infty}
+\frac{r}{\,k - r - \alpha\,}\,\|R_t - (R_t)_r\|_{1}.
\]
Since \(D_t\) is purely diagonal, this controls the off–diagonal entries of \(A - (D_t + U_tU_t^T)\).

\textbf{Diagonal error (Section 9.2 Diag++ \cite{baston2022stochastic}).}  
From the Diag\(++\) analysis, with \(s\) Rademacher queries one obtains
\[
\Pr\Bigl(\|D_t - \diag(R_t)\|_{2}
\;\le\;
\varepsilon\,\|\diag(R_t)\|_{2}\Bigr)\;\ge\;1-\delta,
\]
provided 
\[
s \;>\; \frac{4\,\tr(R_t)}{\|\diag(R_t)\|_{2}}
        \frac{\sqrt{\ln(2n/\delta)}}{\varepsilon} \;+\; c\ln(1/\delta).
\]
This controls the diagonal entries of \(A - (D_t + U_tU_t^T)\).

\textbf{Combine.}  
With probability at least \(1-2\delta\) both bounds hold simultaneously, and
\[
\|A - (D_t + U_tU_t^T)\|_{\infty}
= \max\Bigl\{\|R_t - U_tU_t^T\|_{\infty},\,
             \|D_t - \diag(R_t)\|_{\infty}\Bigr\}
\;\le\;
E_{\rm lr} + E_{\rm diag}.
\]
This completes the proof.
\end{proof}

\begin{rmk}
    While Theorem~\ref{thm:stochbound} is not a convergence proof for Algorithm~\ref{alg:stoch_alt_lr_diag},
    it helps explain its strong asymptotic performance. Near a fixed point corresponding to an LRPD
    decomposition with rank-$r$ low-rank component, the residual $R_t = A - D_{t-1}$ is already well
    approximated by its best rank-$r$ truncation $(R_t)_r$. Consequently, the dominant term
    $\|R_t - U_tU_t^\top\|_\infty$ in the error bound is controlled by
    $\|R_t - (R_t)_r\|$, up to the Nyström approximation factors in $E_{\rm lr}$. When
    $\|R_t - (R_t)_r\|$ is small near a fixed point,
    the additional stochastic error is negligible, and the stochastic iterate closely tracks
    the deterministic Alt update.
\end{rmk}

We illustrate a simple stochastic version in Algorithm~\ref{alg:stoch_alt_lr_diag}, while a more robust implementation of the Nyström sketch can be found in Algorithm~3 of~\cite{tropp2017fixed}.

\begin{algorithm}
\caption{Stochastic Alternating Low‐Rank then Diagonal (Stochastic Alt)}
\label{alg:stoch_alt_lr_diag}
\begin{algorithmic}[1]
\REQUIRE Symmetric matrix $A\in\mathbb{R}^{n\times n}$, target rank $k$, iterations $T$, total mat-vec budget $b$
\ENSURE Approximate decomposition $M = D + U U^\top$

\STATE $D \gets 0_{\,n\times n}$ \COMMENT{Initialize diagonal to zero}
\STATE $\texttt{diagA} \gets \diag(A)$
\IF{$b \le k$}
  \RETURN \textsc{Alt}$(A, k, T)$
\ENDIF
\STATE $m_1 \gets \lfloor \tfrac{2}{3}b \rfloor$
\FOR{$t = 1$ \TO $T$}
  \STATE $R \gets A - D$
  \STATE $s \gets \max(m_1, k + 1)$
  \STATE Draw random Gaussian matrix $\Omega \in \mathbb{R}^{n \times s}$
  \STATE $Y \gets R \Omega$, \quad $C \gets \Omega^\top Y$
  \STATE Compute eigendecomposition $C = W \Lambda W^\top$
  \STATE Sort eigenpairs: $\lambda_1 \ge \cdots \ge \lambda_s$
  \IF{$\lambda_k \le 10^{-12} \lambda_1$}
    \STATE $f_{\text{eff}} \gets \#\{i : \lambda_i > 10^{-12} \lambda_1\}$
    \STATE $\texttt{inv\_sqrt}_i \gets 1/\sqrt{\max(\lambda_i,0) + 10^{-16}},\quad i=1,\dots,f_{\text{eff}}$
    \STATE $U_{[:,\,1:f_{\text{eff}}]} \gets Y W_{[:,\,1:f_{\text{eff}}]} \cdot \texttt{inv\_sqrt}^\top$
    \STATE $\texttt{diagU}_i \gets \sum_{j=1}^{f_{\text{eff}}} U_{i,j}^2,\quad i=1,\dots,n$
    \STATE $d_i \gets \max(\texttt{diagA}_i - \texttt{diagU}_i, 0)$
    \STATE $D \gets \diag(d_1,\dots,d_n)$
    \STATE \textbf{break}
  \ENDIF
  \STATE $\texttt{inv\_sqrt}_i \gets 1/\sqrt{\max(\lambda_i,0) + 10^{-16}},\quad i=1,\dots,s$
  \STATE $U \gets Y W \cdot \texttt{inv\_sqrt}^\top$
  \STATE $U \gets U_{[:,\,1:k]}$
  \STATE $\texttt{diagU}_i \gets \sum_{j=1}^{k} U_{i,j}^2,\quad i=1,\dots,n$
  \STATE $d_i \gets \max(\texttt{diagA}_i - \texttt{diagU}_i, 0)$
  \STATE $D \gets \diag(d_1,\dots,d_n)$
\ENDFOR
\STATE $M \gets D + UU^\top$
\RETURN $M$
\end{algorithmic}
\end{algorithm}

Figures~\ref{fig:alt_lr_diag_stochastic_error_iteration} and~\ref{fig:alt_lr_diag_stochastic_n} illustrate the performance of the proposed stochastic Alt algorithm. In the first figure, we compare the Frobenius error of the stochastic and deterministic variants and observe that the stochastic version performs comparably to the deterministic one. Figure~\ref{fig:alt_lr_diag_stochastic_n} demonstrates randomized LRPD recovery on a $150\times150$ matrix with only 30 mat–vec products per iteration, showing successful recovery of the true rank-8 low-rank component.

\begin{figure}[H]
    \centering
    \includegraphics[width=0.5\textwidth]{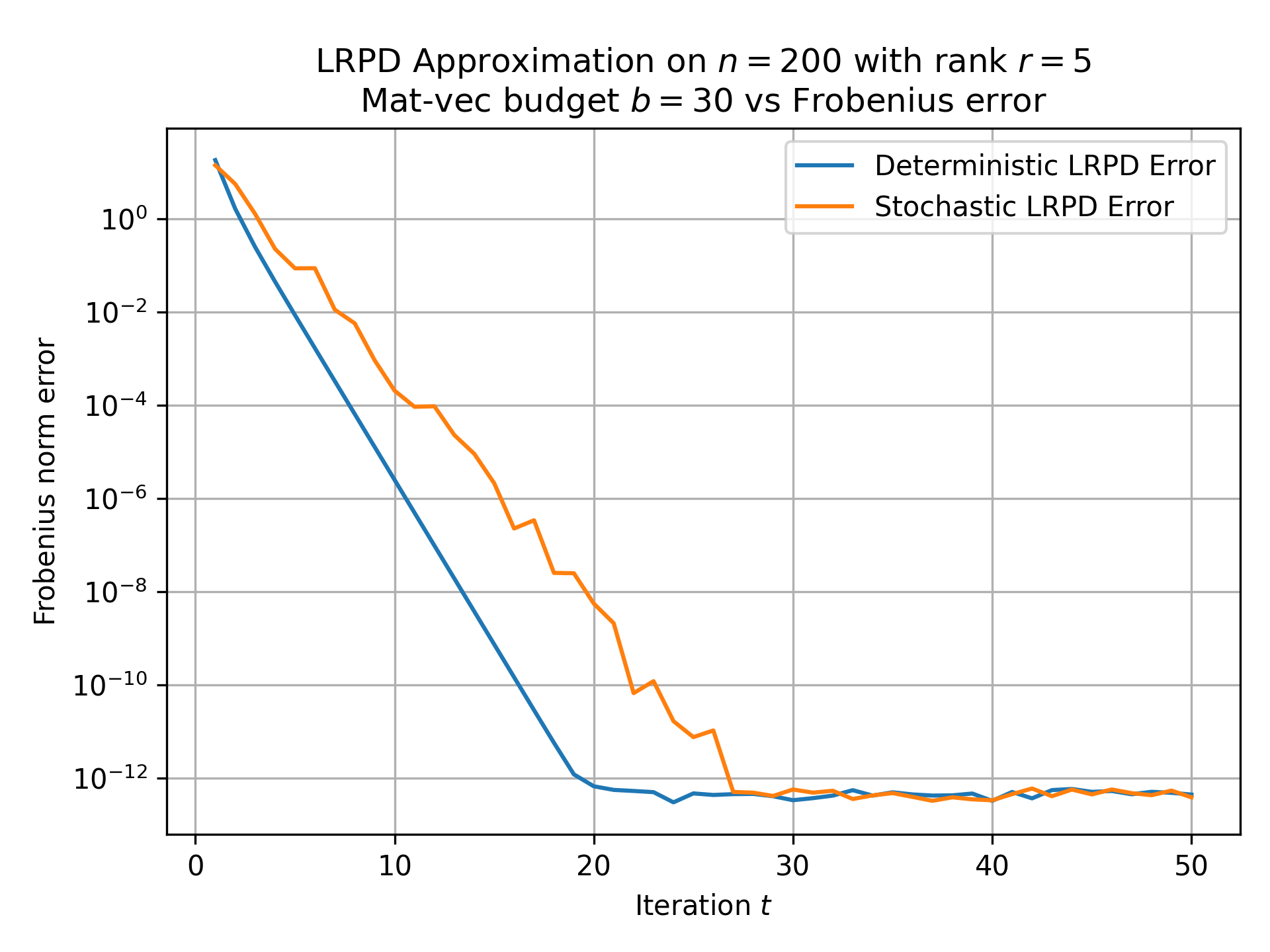}
    \caption{Randomized and deterministic Alt comparison on a $200\times200$ matrix with target rank $r=5$, using 30 matrix–vector products per iteration.}
    \label{fig:alt_lr_diag_stochastic_error_iteration}
\end{figure}

\begin{figure}[H]
    \centering
    \includegraphics[width=0.5\textwidth]{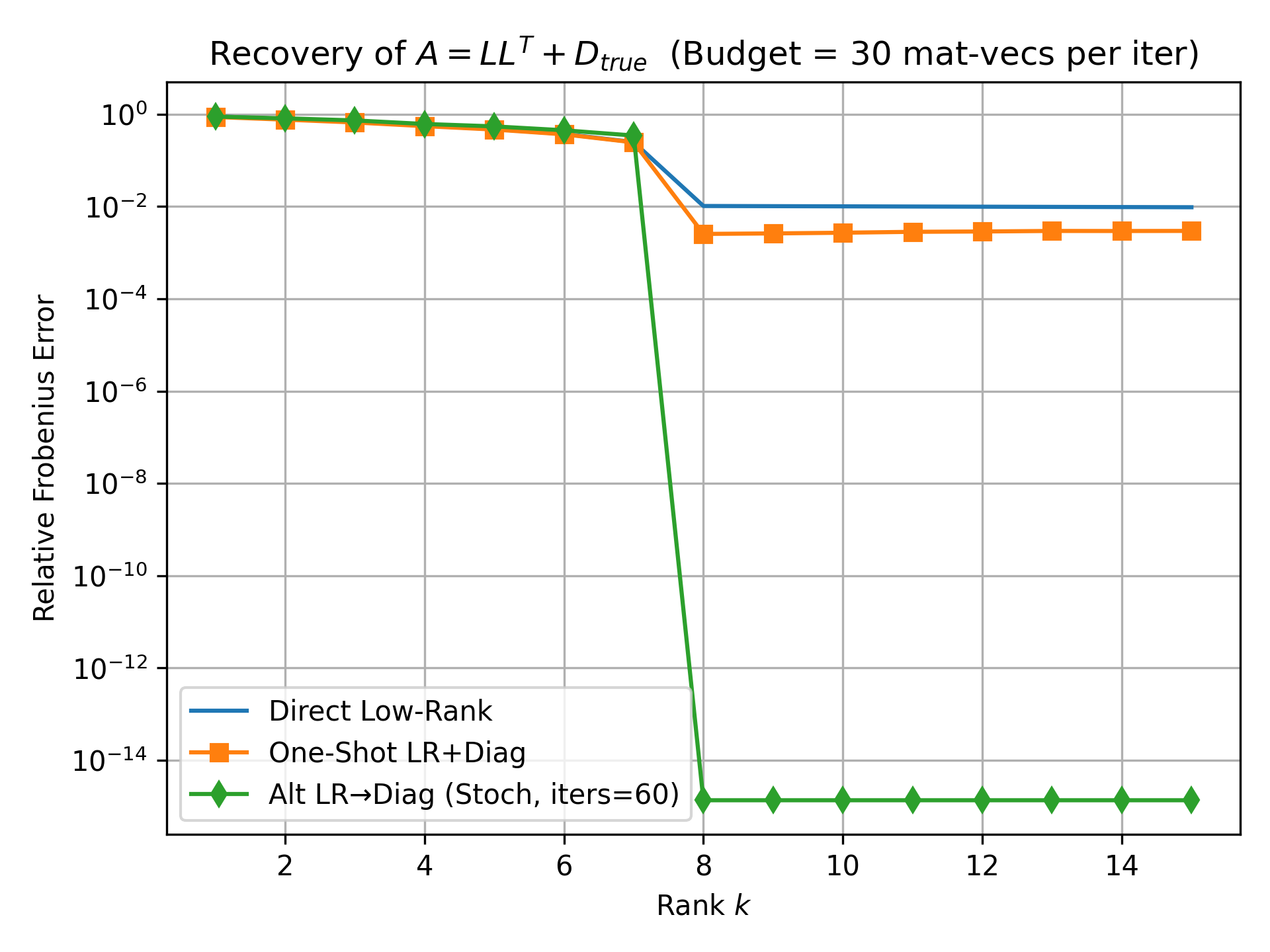}
    \caption{Randomized Alt recovery on a 150×150 matrix using 30 matrix–vector products per iteration.}
    \label{fig:alt_lr_diag_stochastic_n}
\end{figure}

\section{Applications}

To evaluate the effectiveness of structured diagonal updates, we apply our method to empirical stock covariance data derived from daily returns of the S\&P 500. Using $k$-means clustering to identify groups of stocks with similar behavior ($k=5$ clusters), we enforce a block-diagonal structure on the diagonal correction step. Figure~\ref{fig:stock_dlr} compares several variants of Alt, including uniform block splitting, cluster-aligned blocks, and diagonal-only updates, showing the clear advantage of adapting to cluster structure. The block size is implicitly determined by cluster membership. Notably, when the low-rank component has dimension $k=6$, the clustered-block variant already achieves machine precision, while all methods outperform the simple low-rank approximation at smaller ranks. As the matrix is not exactly low-rank, the pure low-rank error remains significant up to $k=29$; only when $k=30$ (the full rank) does the low-rank method reach machine precision. By Theorem~\ref{thm:suffcontraction}, all variants are guaranteed to perform at least as well as low-rank alone, and we observe that methods with structured diagonal updates consistently exceed this theoretical baseline.

A block--diagonal diagonal correction means that instead of updating all diagonal
entries independently, we partition the index set $\{1,\dots,n\}$ into disjoint blocks
$B_1,\dots,B_m$ and restrict the update of $D$ to the block structure
\[
D=\operatorname{blkdiag}(D_{B_1},\dots,D_{B_m}),
\qquad
D_{B_\ell}\in\mathbb{R}^{|B_\ell|\times|B_\ell|},
\]
with each block updated only from the corresponding principal submatrix.  In practice this
enforces that variables within the same cluster (e.g.\ stocks exhibiting similar return
patterns) share a joint diagonal correction, while different clusters remain decoupled.  The
alternating scheme remains unchanged: each iteration still performs a low--rank projection
$U\leftarrow T_k(A-D)$ followed by a block diagonal update for $D$.  The only modification is
that the functional $\operatorname{diag}(\cdot)$ is replaced by a blockwise spectral projection,
i.e.\ $D_{B_\ell}\leftarrow\bigl(A_{B_\ell B_\ell}-U U^\top_{B_\ell B_\ell}\bigr)_+$.  Thus the algorithm preserves the
same alternating low--rank/diagonal structure, but the diagonal update is to operate
on blocks rather than individual coordinates.  This allows the method to capture within--cluster
heteroscedasticity while still exploiting global low--rank structure.

\begin{figure}[H]
    \centering
    \begin{minipage}{0.48\textwidth}
        \centering
        \includegraphics[width=\linewidth]{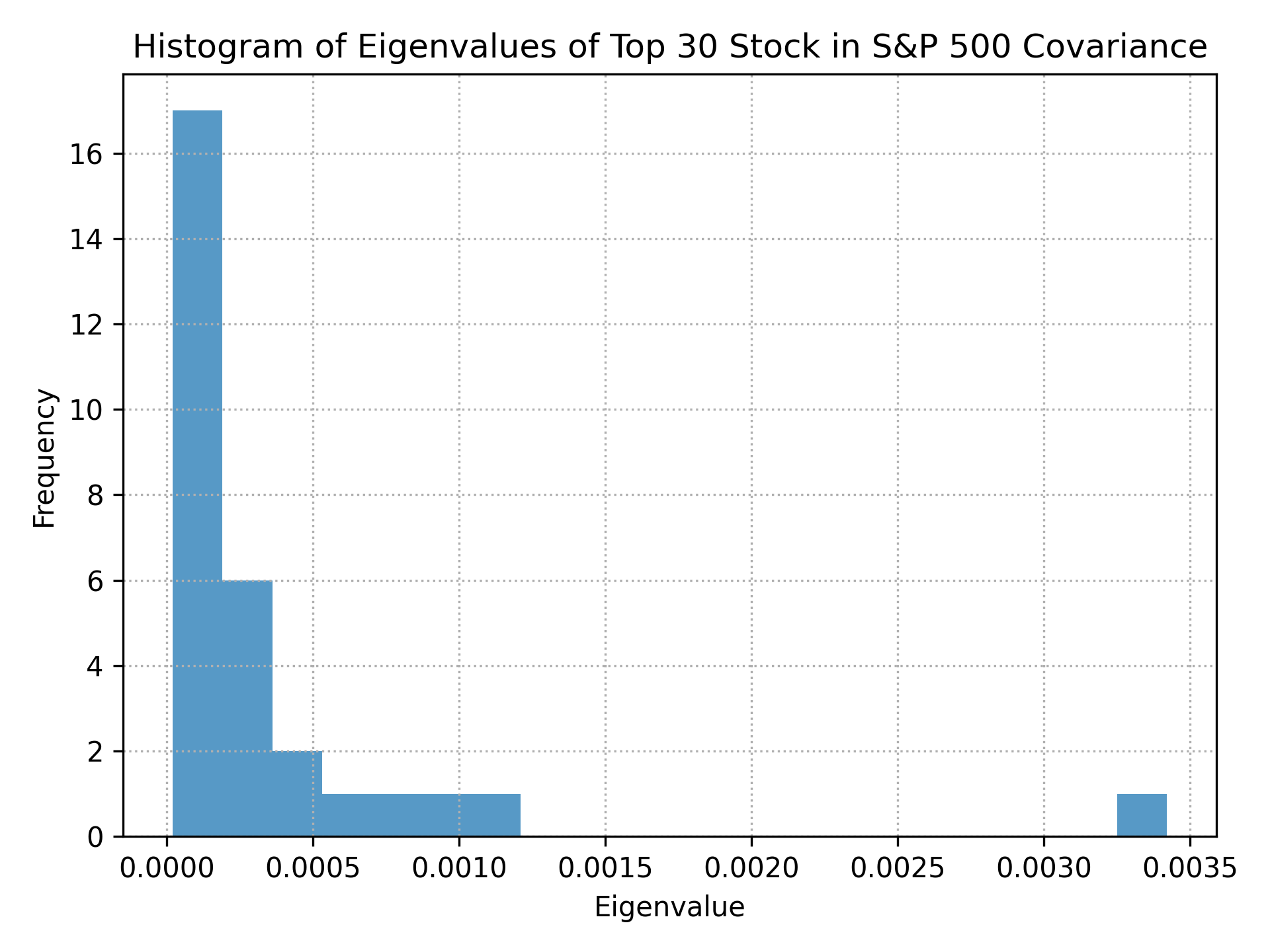}
        \caption{Histogram of eigenvalues of the sample covariance matrix of the top 30 S\&P 500 stocks (daily returns over two years).}
        \label{fig:eig_hist_spy500}
    \end{minipage}\hfill
    \begin{minipage}{0.48\textwidth}
        \centering
        \includegraphics[width=\linewidth]{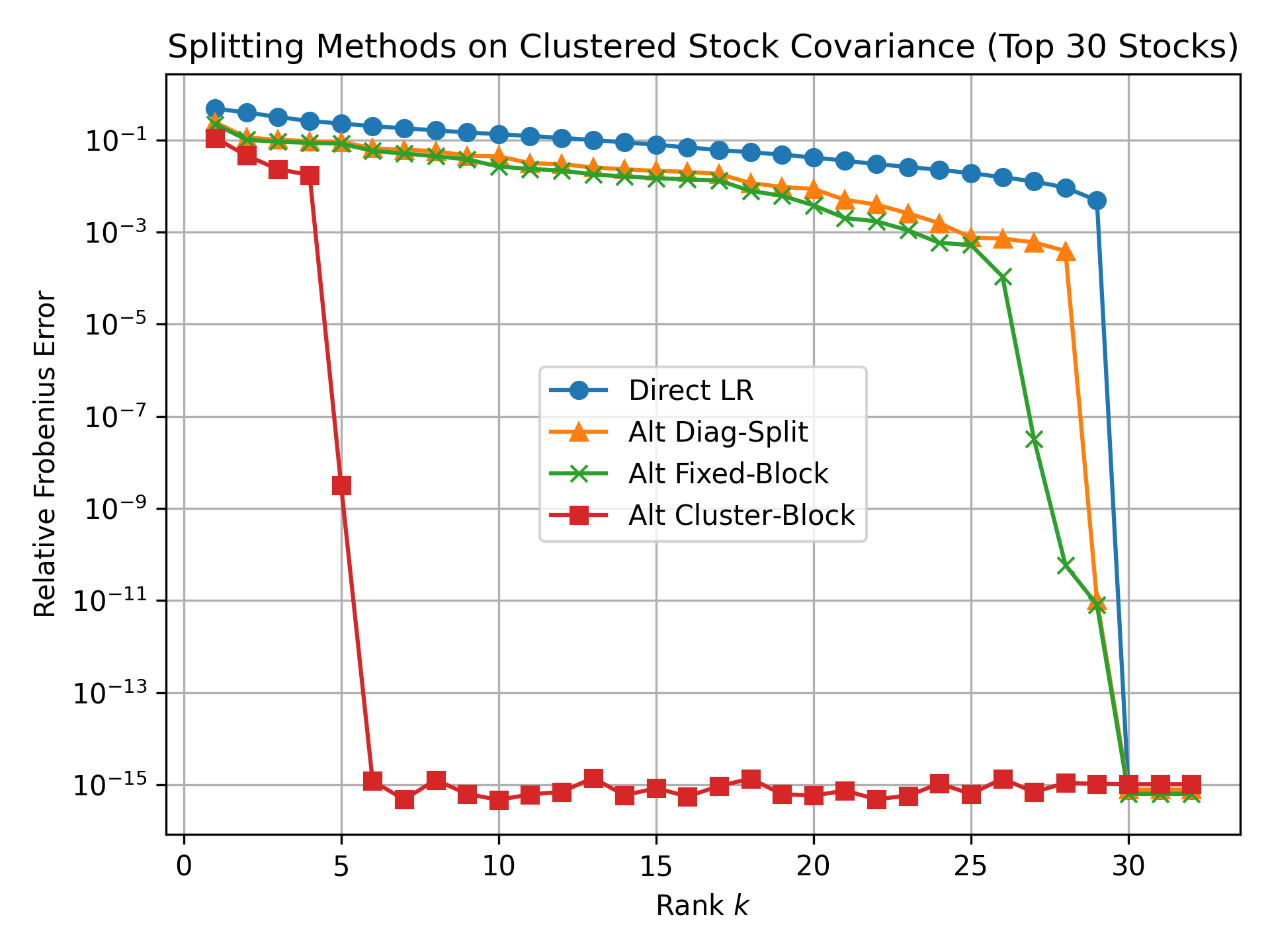}
        \caption{S\&P 500 LRPD decomposition after K-means clustering (number of clusters = 5).}
        \label{fig:stock_dlr}
    \end{minipage}
\end{figure}

In Figure~\ref{fig:eig_hist_spy500}, the histogram shows that most eigenvalues are small and concentrated near zero, while a few eigenvalues stand out as significantly larger. These outliers correspond to dominant factors in the market. In particular, the largest eigenvalue is much larger than the others, and its associated eigenvector has entries of the same sign, indicating that a strong positive correlation structure can explain much of the variation in these stocks.


\section{Gradient descent iterations for low-rank plus diagonal}

We have seen throughout the paper that structured alternating updates, such as Alt and its
randomized variants, enjoy rapid and stable convergence with essentially no tuning. It is therefore
natural to ask how these methods compare with approaches based on gradient flow or higher-order
optimization applied directly to the Gaussian log-likelihood at a given covariance
$\widehat{\Sigma}$, as explored for related low-rank and structured covariance models in, e.g.,
\cite{lawley1940estimation, anderson2003introduction, saunderson2012diagonal, sun2016majorization, zhou2022covariance}.
At first glance, such methods appear attractive since the objective is smooth on the positive
definite cone and the LRPD parametrization $M(D,U)=D+UU^\top$ is simple. However, as we show below,
the resulting gradient flow is extremely stiff. The map $(D,U)\mapsto f(D,U)$ has a highly
ill-conditioned Hessian whenever the spectrum of $M$ is spread out, and the updates strongly couple
$U$ and $D$ in a manner that forces very small step sizes. As a result, naive gradient-based or
Newton-type methods suffer from slow convergence and numerical instability unless carefully
regularized or tuned, in sharp contrast with the robustness of the alternating spectral scheme.

To make this comparison precise, we consider gradient descent applied directly to the Gaussian
negative log-likelihood
\[
f(D,U) = \log\det M + \mathrm{tr}(M^{-1}\widehat\Sigma),
\quad M(D,U)=D+UU^\top.
\]
The explicit gradients with respect to $U$ and the diagonal $d=\mathrm{diag}(D)$ are derived in
Appendix~\ref{app:gradients} and are given by
\begin{align}
\nabla_U f &= 2\,(M^{-1}-M^{-1}\widehat\Sigma\,M^{-1})\,U, \label{eq:gradientU_main} \\
\frac{\partial f}{\partial d_j} &= \bigl[M^{-1}-M^{-1}\widehat\Sigma\,M^{-1}\bigr]_{jj},
\quad j=1,\ldots,n. \label{eq:gradientdiag_main}
\end{align}
Applying a constant step size $\eta>0$ yields the gradient descent iterations
\begin{align*}
U^{(t+1)} &= U^{(t)} - \eta\,\nabla_U f\bigl(D^{(t)},U^{(t)}\bigr), \\
d^{(t+1)}_j &= d^{(t)}_j - \eta\,\frac{\partial f}{\partial d_j}
\bigl(D^{(t)},U^{(t)}\bigr), \quad j=1,\ldots,n.
\end{align*}

While straightforward to implement, gradient descent for this problem is highly sensitive to the
choice of step size~$\eta$ and initialization. Small step sizes lead to very slow convergence,
whereas overly large step sizes may cause instability or divergence, a well-known phenomenon in
ill-conditioned covariance problems. Figure~\ref{fig:gd_vs_altlrdiag} illustrates this behavior:
panel~(a) compares convergence under random and SVD-based initialization, showing that careful
initialization improves stability but does not eliminate slow convergence. In contrast, panel~(b)
demonstrates that Alt, which alternates low-rank projection with diagonal corrections, achieves
significantly faster and more stable error reduction without delicate tuning of~$\eta$.

\begin{figure}[H]
    \centering
    \begin{minipage}{0.49\textwidth}
        \centering
        \includegraphics[width=\linewidth]{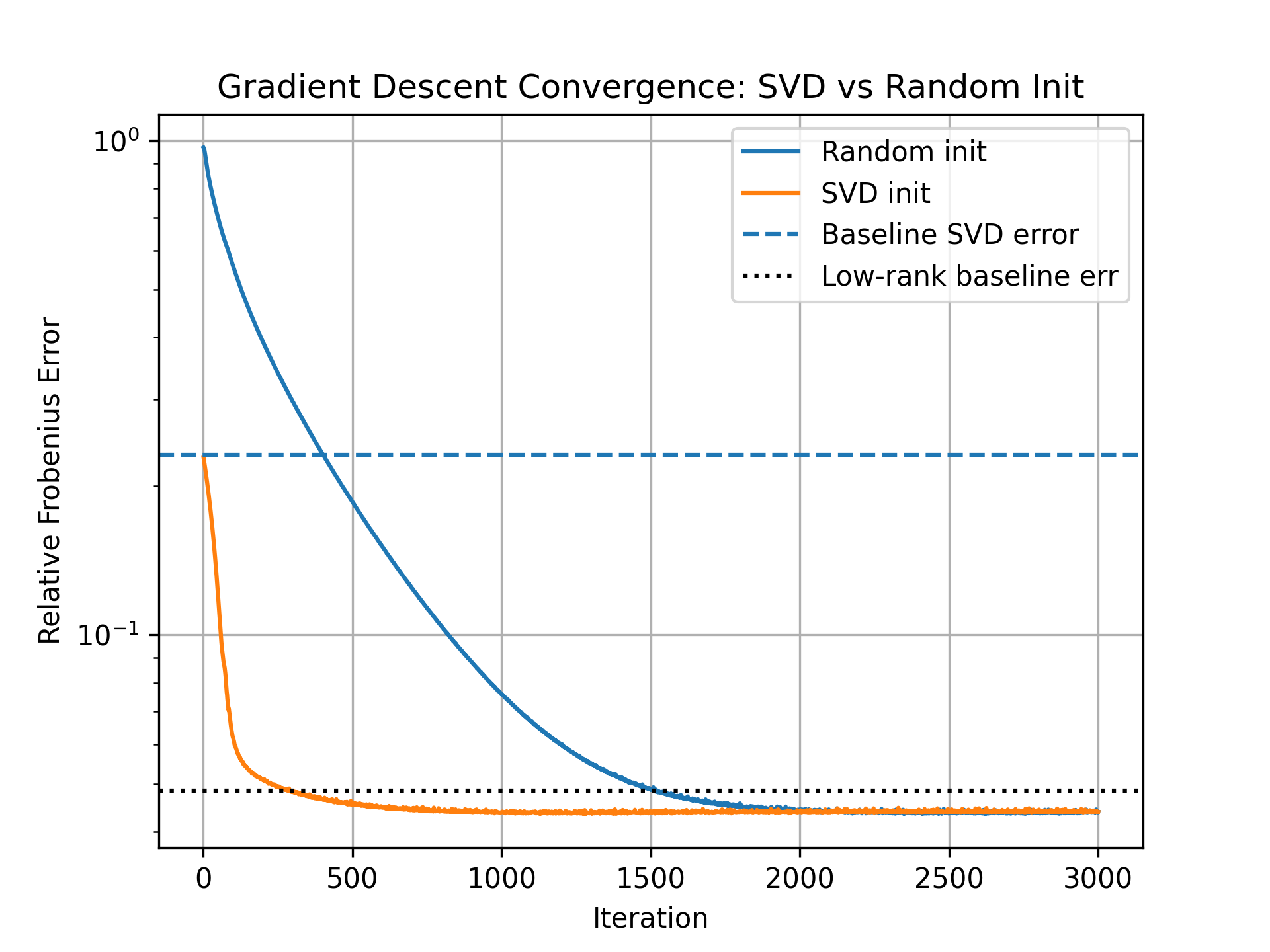}
        \caption*{(a) Gradient descent convergence with and without SVD initialization}
    \end{minipage}\hfill
    \begin{minipage}{0.49\textwidth}
        \centering
        \includegraphics[width=\linewidth]{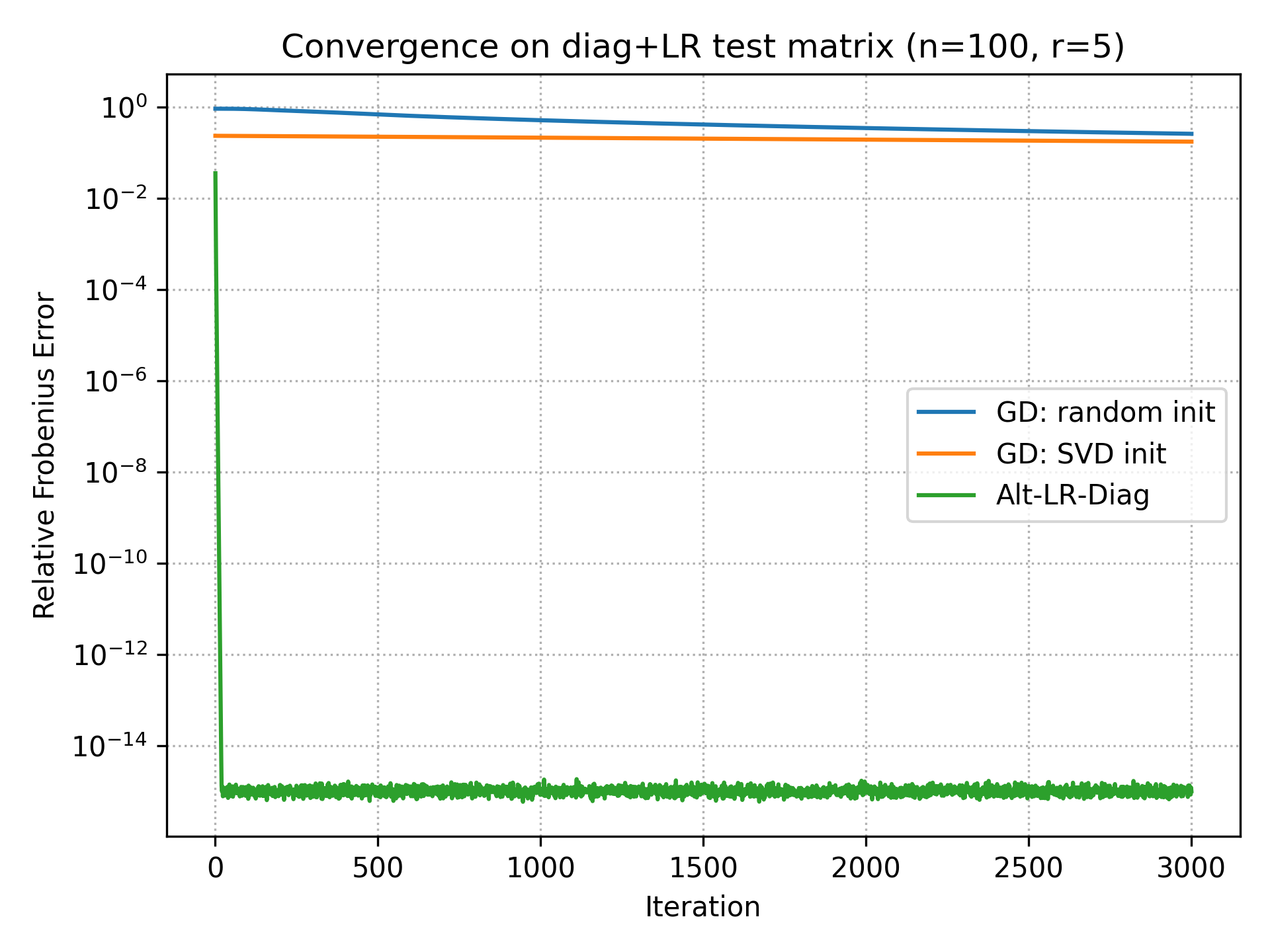}
        \caption*{(b) Alt vs. gradient descent}
    \end{minipage}
    \caption{Comparison of optimization strategies for LRPD decomposition on a synthetic matrix with diagonal plus rank-5 structure. Gradient descent uses a fixed step size $\eta = 10^{-2}$ and converges under both random and SVD initialization, whereas Alt achieves significantly better performance with simpler alternating updates.}
    \label{fig:gd_vs_altlrdiag}
\end{figure}

\section*{Acknowledgment}
 This work was supported by the U.S. Department of Energy, Office of Science (SC), Advanced Scientific Computing Research (ASCR), Competitive Portfolios Project on Energy Efficient Computing: A Holistic Methodology, under Contract DE-AC02-06CH11357. The authors thanks Cong Ma for discussions and feedback.

\bibliographystyle{siamplain}
\bibliography{refs}

@article{bonnabel2024low,
  title={Low-rank plus diagonal approximations for Riccati-like matrix differential equations},
  author={Bonnabel, Silv{\`e}re and Lambert, Marc and Bach, Francis},
  journal={SIAM Journal on Matrix Analysis and Applications},
  volume={45},
  number={3},
  pages={1669--1688},
  year={2024},
  publisher={SIAM}
}

@incollection{parshakova2024factor,
  title={Factor fitting, rank allocation, and partitioning in multilevel low rank matrices},
  author={Parshakova, Tetiana and Hastie, Trevor and Darve, Eric and Boyd, Stephen},
  booktitle={Optimization, Discrete Mathematics and Applications to Data Sciences},
  pages={135--173},
  year={2024},
  publisher={Springer}
}

@article{saunderson2012diagonal,
  title={Diagonal and low-rank matrix decompositions, correlation matrices, and ellipsoid fitting},
  author={Saunderson, James and Chandrasekaran, Venkat and Parrilo, Pablo A and Willsky, Alan S},
  journal={SIAM Journal on Matrix Analysis and Applications},
  volume={33},
  number={4},
  pages={1395--1416},
  year={2012},
  publisher={SIAM}
}

@inproceedings{zhao2016low,
  title={Low-rank plus diagonal adaptation for deep neural networks},
  author={Zhao, Yong and Li, Jinyu and Gong, Yifan},
  booktitle={2016 IEEE International Conference on Acoustics, Speech and Signal Processing (ICASSP)},
  pages={5005--5009},
  year={2016},
  organization={IEEE}
}

@article{stein2014limitations,
  title={Limitations on low rank approximations for covariance matrices of spatial data},
  author={Stein, Michael L},
  journal={Spatial Statistics},
  volume={8},
  pages={1--19},
  year={2014},
  publisher={Elsevier}
}

@article{zhou2022covariance,
  title={Covariance matrix estimation under low-rank factor model with nonnegative correlations},
  author={Zhou, Rui and Ying, Jiaxi and Palomar, Daniel P},
  journal={IEEE Transactions on Signal Processing},
  volume={70},
  pages={4020--4030},
  year={2022},
  publisher={IEEE}
}

@article{sun2016majorization,
  title={Majorization-minimization algorithms in signal processing, communications, and machine learning},
  author={Sun, Ying and Babu, Prabhu and Palomar, Daniel P},
  journal={IEEE Transactions on Signal Processing},
  volume={65},
  number={3},
  pages={794--816},
  year={2016},
  publisher={IEEE}
}

@article{baston2022stochastic,
  title={Stochastic diagonal estimation: probabilistic bounds and an improved algorithm},
  author={Baston, Robert A and Nakatsukasa, Yuji},
  journal={arXiv preprint arXiv:2201.10684},
  year={2022}
}

@article{frangella2023randomized,
  title={Randomized nystr{\"o}m preconditioning},
  author={Frangella, Zachary and Tropp, Joel A and Udell, Madeleine},
  journal={SIAM Journal on Matrix Analysis and Applications},
  volume={44},
  number={2},
  pages={718--752},
  year={2023},
  publisher={SIAM}
}

@article{tropp2017fixed,
  title={Fixed-rank approximation of a positive-semidefinite matrix from streaming data},
  author={Tropp, Joel A and Yurtsever, Alp and Udell, Madeleine and Cevher, Volkan},
  journal={Advances in Neural Information Processing Systems},
  volume={30},
  year={2017}
}

@article{szlam2014implementation,
  title={An implementation of a randomized algorithm for principal component analysis},
  author={Szlam, Arthur and Kluger, Yuval and Tygert, Mark},
  journal={arXiv preprint arXiv:1412.3510},
  year={2014}
}

@book{horn2012matrix,
  title={Matrix analysis},
  author={Horn, Roger A and Johnson, Charles R},
  year={2012},
  publisher={Cambridge university press}
}

@article{davis1970rotation,
  title={The rotation of eigenvectors by a perturbation. III},
  author={Davis, Chandler and Kahan, William Morton},
  journal={SIAM Journal on Numerical Analysis},
  volume={7},
  number={1},
  pages={1--46},
  year={1970},
  publisher={SIAM}
}

@book{higham2008functions,
  title={Functions of matrices: theory and computation},
  author={Higham, Nicholas J},
  year={2008},
  publisher={SIAM}
}

@article{ledoit2004well,
  title={A well-conditioned estimator for large-dimensional covariance matrices},
  author={Ledoit, Olivier and Wolf, Michael},
  journal={Journal of Multivariate Analysis},
  volume={88},
  number={2},
  pages={365--411},
  year={2004}
}

@article{bickel2008regularized,
  title={Regularized estimation of large covariance matrices},
  author={Bickel, Peter J and Levina, Elizaveta},
  journal={Annals of Statistics},
  volume={36},
  number={1},
  pages={199--227},
  year={2008}
}

@article{lawley1940factor,
  title={The estimation of factor loadings by the method of maximum likelihood},
  author={Lawley, DN},
  journal={Proceedings of the Royal Society of Edinburgh},
  volume={60},
  pages={64--82},
  year={1940}
}

@book{anderson2003introduction,
  title={An Introduction to Multivariate Statistical Analysis},
  author={Anderson, T. W.},
  edition={3rd},
  publisher={Wiley},
  year={2003}
}

@article{fan2013large,
  title={Large covariance estimation by thresholding principal orthogonal complements},
  author={Fan, Jianqing and Liao, Yuan and Mincheva, Martina},
  journal={Journal of the Royal Statistical Society: Series B},
  volume={75},
  number={4},
  pages={603--680},
  year={2013}
}

@article{johnstone2001distribution,
  title={On the distribution of the largest eigenvalue in principal components analysis},
  author={Johnstone, Iain M.},
  journal={Annals of Statistics},
  volume={29},
  number={2},
  pages={295--327},
  year={2001}
}

@article{lawley1940estimation,
  title={The estimation of factor loadings by the method of maximum likelihood},
  author={Lawley, D. N.},
  journal={Proceedings of the Royal Society of Edinburgh},
  volume={60},
  pages={64--82},
  year={1940}
}

\newpage
\appendix

\section{Derivation of gradients for the LRPD log-likelihood}
\label{app:gradients}

In this appendix we derive the gradients of the Gaussian negative log-likelihood
\[
f(D,U) = \log\det M + \mathrm{tr}(M^{-1}\widehat\Sigma),
\quad M(D,U)=D+UU^\top,
\]
with respect to the low-rank factor $U$ and the diagonal $d=\mathrm{diag}(D)$.

First, observe that the differential of $M$ decomposes as
\[
\mathrm{d}M = \mathrm{d}D + (\mathrm{d}U)\,U^T + U\,(\mathrm{d}U)^T.
\]
Using standard matrix differentiation identities,
\[
\mathrm{d}\log\det M = \mathrm{tr}(M^{-1}\,\mathrm{d}M),
\quad
\mathrm{d}(M^{-1}) = -M^{-1}(\mathrm{d}M)M^{-1},
\]
we compute
\begin{align*}
\mathrm{d}f
&= \mathrm{tr}(M^{-1}\,\mathrm{d}M)
+ \mathrm{tr}\bigl(-M^{-1}(\mathrm{d}M)M^{-1}\widehat\Sigma\bigr) \\
&= \mathrm{tr}\Bigl[(M^{-1}-M^{-1}\widehat\Sigma\,M^{-1})\,\mathrm{d}M\Bigr].
\end{align*}
Substituting the expression for $\mathrm{d}M$ yields
\[
\mathrm{d}f
= \mathrm{tr}\bigl[(M^{-1}-M^{-1}\widehat\Sigma\,M^{-1})
((\mathrm{d}U)U^T + U(\mathrm{d}U)^T)\bigr].
\]
By cyclicity of the trace and symmetry of the factor, this simplifies to
\[
\mathrm{d}f
= 2\,\mathrm{tr}\Bigl[(M^{-1}-M^{-1}\widehat\Sigma\,M^{-1})
\,U\,(\mathrm{d}U)^T\Bigr].
\]
Identifying $\mathrm{d}f=\mathrm{tr}((\nabla_U f)^T\mathrm{d}U)$ gives
\[
\nabla_U f = 2\,(M^{-1}-M^{-1}\widehat\Sigma\,M^{-1})\,U.
\]

For the diagonal component, only $\mathrm{d}D=\mathrm{diag}(\mathrm{d}d)$ contributes, yielding
\[
\mathrm{d}f
= \sum_{j=1}^n
\bigl[M^{-1}-M^{-1}\widehat\Sigma\,M^{-1}\bigr]_{jj}\,\mathrm{d}d_j,
\]
and hence
\[
\frac{\partial f}{\partial d_j}
= \bigl[M^{-1}-M^{-1}\widehat\Sigma\,M^{-1}\bigr]_{jj},
\quad j=1,\ldots,n.
\]

\vspace{0.1cm}
\begin{flushright}
	\scriptsize \framebox{\parbox{2.5in}{Government License: The
			submitted manuscript has been created by UChicago Argonne,
			LLC, Operator of Argonne National Laboratory (``Argonne").
			Argonne, a U.S. Department of Energy Office of Science
			laboratory, is operated under Contract
			No. DE-AC02-06CH11357.  The U.S. Government retains for
			itself, and others acting on its behalf, a paid-up
			nonexclusive, irrevocable worldwide license in said
			article to reproduce, prepare derivative works, distribute
			copies to the public, and perform publicly and display
			publicly, by or on behalf of the Government. The Department of Energy will provide public access to these results of federally sponsored research in accordance with the DOE Public Access Plan. http://energy.gov/downloads/doe-public-access-plan. }}
	\normalsize
\end{flushright}

\end{document}